\documentclass[12pt]{article}
\usepackage{amsfonts, amsmath, amssymb, latexsym, eucal, amscd} 
\usepackage{graphicx,lscape} 
\usepackage{tikz,bbm}
\usetikzlibrary{positioning}

\usepackage{amsthm}
\usepackage{amscd}

\usepackage{cite}

\newtheorem{definition}{Definition}[section]
\newtheorem{theorem}[definition]{Theorem}
\newtheorem{lemma}[definition]{Lemma}

\newtheorem{example}[definition]{Example}
\newtheorem{conjecture}[definition]{Conjecture}
\newtheorem{problem}[definition]{Problem}

\newtheorem{assumption}[definition]{Assumption}
\newtheorem{proposition}[definition]{Proposition}

\begin{document} 

\title{\bf An equitable partition for
 the distance-\\regular graph of  the bilinear forms
}
\author{
Paul Terwilliger and Jason Williford}
\date{}

\maketitle
\begin{abstract} We consider a type of distance-regular graph $\Gamma=(X, \mathcal R)$ called a bilinear forms graph. We assume that the diameter $D$ of $\Gamma$ is at least $3$.
Fix adjacent vertices $x,y \in X$. In our first main result, we introduce an equitable partition of $X$ that has $6D-2$ subsets and the following feature:
for every subset in the equitable partition, the vertices in the subset are equidistant to $x$ and equidistant to $y$.
This equitable partition is called the $(x,y)$-partition of $X$. By definition, the subconstituent algebra $T=T(x)$  is generated by the Bose-Mesner
algebra of $\Gamma$ and the dual Bose-Mesner algebra of $\Gamma$ with respect to $x$.
As we will see, for the $(x,y)$-partition of $X$ the characteristic vectors of the subsets form a basis for
a $T$-module $U=U(x,y)$.  In our second main result, we decompose $U$ into an orthogonal direct sum of irreducible $T$-modules.
This sum has five summands:  the primary $T$-module and four irreducible $T$-modules that have endpoint one. 
 We show that every irreducible $T$-module
with endpoint one is isomorphic to exactly one of the nonprimary summands.
\medskip

 \end{abstract}
 
 \section{Introduction}
 Let $\Gamma=(X,\mathcal R)$ denote a distance-regular graph with diameter $D\geq 3$ (formal definitions begin in Section 2). 
 Throughout this section, fix a vertex $x \in X$.
It can be illuminating to describe $\Gamma$ from the point of view of $x$.
In this description, there is  a  combinatorial perspective and an algebraic perspective.
 In the combinatorial perspective, one uses the path-length distance function $\partial$ as follows.
 There is a partition of $X$ into subsets $\lbrace  \Gamma_i(x) \rbrace_{i=0}^D$ such that $\Gamma_i(x) = \lbrace y \in X \vert \partial (x,y)=i\rbrace $ for $0 \leq i \leq D$.
  This partition is equitable in the sense of \cite[p.~436]{bcn}.
In the algebraic perspective one considers the subconstituent algebra $T=T(x)$; this is generated by the Bose-Mesner algebra of $\Gamma$
and the dual Bose-Mesner algebra of $\Gamma$ with respect to $x$. The algebra $T$ is finite-dimensional, noncommutative, and semisimple.
The Wedderburn decomposition of $T$ is discussed in many works; for a survey see  \cite[Section~13]{int}.
The combinatorial and algebraic perspectives are related by the fact that for
the equitable partition  $\lbrace  \Gamma_i(x) \rbrace_{i=0}^D$, the characteristic vectors of the subsets form a basis for an irreducible $T$-module \cite[Section~7]{int}. This $T$-module is
said to be primary \cite[Definition~7.2]{int}.
\medskip

\noindent  In order to understand the nonprimary irreducible $T$-modules, it can be useful to
fix a vertex $y\in X$ such that $y \not=x$, and
employ an equitable partition of $X$ with the following feature:
\begin{align}
\begin{split}
& \hbox{for every subset in the equitable partition, the vertices}\\ 
&\hbox{in the subset are equidistant to $x$ and equidistant to $y$.}
\end{split}
\label{eq:feature}
\end{align}
Such equitable partitions have been employed in the following situations:

\bigskip
\begin{tabular}[t]{ccc}
$\partial(x,y)$ & $\Gamma$ description & citations
 \\
 \hline
$1$ & $1$-homogeneous& \cite{1hom, hom}\\
$1$ & tight & \cite{jtgo, jkt} \\  
$1$ &$Q$-polynomial and $a_1=0$ & \cite{mik1, mik5} \\
$1$ &classical parameters of negative type &\cite{mik2, mik6} \\
$1$ &pseudo 1-homogeneous &\cite{1hom, jt2}  \\
$2$& bipartite and $2$-homogeneous &\cite{2HomCurt, 2homT}\\
$2$ &bipartite $Q$-polynomial and $c_2=1$ &\cite{maclean,mik4} \\
$2$ &bipartite and $c_2=2$ &\cite{penjic} \\
$2$ &bipartite dual polar graph &\cite{mik7} \\
$D$ & bipartite $Q$-polynomial &\cite{caughman1, caughman2}  \\
$D$ &supports a spin model & \cite{CN2H,nomSpinModel} 
    \end{tabular}

\bigskip

\noindent In the present paper, we assume that $\Gamma$ is a bilinear forms graph \cite[p.~280]{bcn}.
We fix a vertex $y \in X$ such that $\partial(x,y)=1$. In our first main result, we introduce  an equitable partition of $X$ that has $6D-2$ subsets and the feature \eqref{eq:feature}.
This equitable partition is called the $(x,y)$-partition of $X$. As we will see, for the $(x,y)$-partition of $X$ the characteristic vectors of the subsets form a basis for
a $T$-module $U=U(x,y)$.  In our second main result, we decompose $U$ into an orthogonal direct sum of irreducible $T$-modules.
This sum has five summands:  the primary $T$-module and four irreducible $T$-modules that have endpoint one in the sense of \cite[Section~13]{int}. We show that every irreducible $T$-module
with endpoint one is isomorphic to exactly one of the nonprimary summands. We obtain this last result using the Hobart/Ito classification \cite{hobart}
of the irreducible $T$-modules  that have endpoint one. 
\medskip

\noindent This paper is organized as follows. Section 2 contains some preliminaries.
Section 3 contains some definitions and basic facts about distance-regular graphs.
Sections 4, 5 contain some definitions and basic facts about the bilinear forms graph $\Gamma=(X,\mathcal R)$.
In Section 6, for adjacent $x,y \in X$ we introduce the $(x,y)$-partition of $X$ and describe it from three points of view.
In Section 7, we prove that the $(x,y)$-partition of $X$ is equitable.
In Section 8, we first use the $(x,y)$-partition of $X$ to obtain a module $U=U(x,y)$ for $T=T(x)$. We then
decompose $U$ into an orthogonal direct sum of irreducible $T$-modules.
In Section 9, we make some comments about $U$.
In Section 10, we give a conjecture and some open problems.
\medskip

\noindent The main results of the paper are
Propositions \ref{prop:parts}, \ref{prop:part2}, \ref{prop:part3}, \ref{prop:equit} and Theorem \ref{thm:Udecomp}.

 \section{Preliminaries}
 In this section, we review some linear algebra  concepts that will be used in the main body of  the paper. 
 Let $\mathbb C$ denote the complex number field.
 For positive integers $m,n$ let the set ${\rm Mat}_{m\times n}(\mathbb C)$ consist of the
 $m \times n $ matrices that have all entries in $\mathbb C$.
 \medskip
 
 \noindent Let $V$ denote an $n$-dimensional vector space over $\mathbb C$. Let
 $\lbrace u_i \rbrace_{i=1}^n$ and $\lbrace v_i \rbrace_{i=1}^n$ denote bases for $V$.
  Define $H \in {\rm Mat}_{n\times n}(\mathbb C)$ such that $v_j = \sum_{i=1}^n H_{i,j} u_i$ for $1 \leq j \leq n$.
 We call $H$ the {\it transition matrix from $\lbrace u_i \rbrace_{i=1}^n$ to $\lbrace v_i \rbrace_{i=1}^n$}.
 \medskip
 
\noindent Let $U$ denote an $n$-dimensional vector space over $\mathbb C$, with a basis $\lbrace u_i \rbrace_{i=1}^n$.
Let $V$ denote an $m$-dimensional vector space over $\mathbb C$, with a basis $\lbrace v_i \rbrace_{i=1}^m$.
Given a $\mathbb C$-linear map $A: U \to V$, define $B \in {\rm Mat}_{m\times n} (\mathbb C)$ such
 that $A u_j = \sum_{i=1}^m B_{i,j} v_i$ for $1 \leq j \leq n$. We call $B$ the {\it matrix that represents $A$ with
 respect to $\lbrace u_i \rbrace_{i=1}^n$ and $\lbrace v_i \rbrace_{i=1}^m$}.
 \medskip
  
   \noindent Let $V$ denote an $n$-dimensional vector space over $\mathbb C$.
 A $\mathbb C$-linear map $A: V \to V$ is called {\it diagonalizable} whenever
 $V$ is spanned by the eigenspaces of $A$. Assume that $A$ is diagonalizable. By an {\it eigenvalue} of $A$
 we mean a root of the minimal polynomial of $A$. For an eigenvalue $\theta$ of $A$, the corresponding {\it primitive idempotent} 
is the $\mathbb C$-linear map $V \to V$ that acts as the identity map on the $\theta$-eigenspace of $A$, and as $0$ on every other eigenspace of $A$.
For an eigenvalue $\theta $ of $A$, the corresponding {\it multiplicity} is the dimension of the $\theta$-eigenspace of $A$.

 \section{Distance-regular graphs}
 
 \noindent In this section, we review some definitions and basic concepts concerning distance-regular graphs.
 More information on this topic can be found in 
 \cite{bbit, bannai, bcn, dkt,int}.
 \medskip
 
 \noindent 
 Let $X$ denote a finite set with $\vert X \vert \geq 2$. An element in $X$ is called a {\it vertex}.
 Let ${\rm Mat}_X(\mathbb C)$ denote the $\mathbb C$-algebra consisting of the 
matrices that have rows and columns indexed by $X$ and all entries in $\mathbb C$. 
Let   $V=\mathbb C^X$ denote the $\mathbb C$-vector space consisting of the column vectors that have coordinates indexed by $X$ and all entries in $\mathbb C$.
The algebra  ${\rm Mat}_X(\mathbb C)$ acts on $V$ by left multiplication. We call $V$ the {\it standard module}.
 We endow $V$ with a Hermitian form $\langle \,,\,\rangle$ such that 
 $\langle u,v\rangle = \overline{u}^t v$ for $u,v \in V$. Here $t$ denotes transpose and $-$ denotes complex-conjugation.
  For $x \in X$ let $\hat x$ denote the vector in $V$ that has $x$-coordinate $1$ and all other
 coordinates 0. The vectors $\lbrace \hat x \vert x \in X\rbrace$ form an orthonormal basis for $V$.
\medskip

\noindent Let 
 $\Gamma=(X, \mathcal R)$ denote an undirected, connected  graph, without loops or multiple edges, with vertex set $X$, adjacency relation $\mathcal R$, and 
 path-length distance function $\partial$. Define $D={\rm max} \lbrace \partial(x,y) \vert x,y \in X\rbrace$ and note that $D\geq 1$.
 We call $D$ the {\it diameter} of $\Gamma$.
For $x \in X$ and $0 \leq i \leq D$ define the set $\Gamma_i(x) = \lbrace y \in X \vert \partial(x,y)=i\rbrace$.
We abbreviate $\Gamma(x) = \Gamma_1(x)$.
For an integer $k\geq 1$ we say that $\Gamma$ is {\it regular with valency $k$} whenever $\vert \Gamma(x) \vert=k$ for all $x \in X$.
 We say that $\Gamma$ is {\it distance-regular} whenever for $0 \leq h,i,j\leq D$ and $x,y \in X$ at distance $\partial(x,y)=h$,
 the integer
 \begin{align*}
 p^h_{i,j} = \vert \Gamma_i(x) \cap \Gamma_j(y) \vert
 \end{align*}
 is independent of $x,y$ and depends only on $h,i,j$. 
 For the rest of this paper, we assume that $\Gamma$ is distance-regular. The parameters $p^h_{i,j}$ $(0\leq h,i,j\leq D)$ are
 called the {\it intersection numbers} of $\Gamma$.
 For notational convenience, we abbreviate
 \begin{align*}
 c_i = p^i_{1,i-1} \; (1 \leq i \leq D),\qquad \quad a_i = p^i_{1,i} \;(0 \leq i \leq D), \qquad \quad b_i = p^i_{1,i+1} \; (0 \leq i \leq D-1).
 \end{align*}
 By construction, $c_1=1$ and $a_0=0$.
 The graph $\Gamma$ is regular  with valency $k=b_0$. Moreover
 \begin{align*}
 c_i + a_i + b_i = k \qquad \qquad (0 \leq i \leq D),
 \end{align*}
where $c_0=0$ and $b_D=0$.
 \medskip
 
 \noindent For $0 \leq i \leq D$, define $k_i = p^0_{i,i}$ and note that $k_i = \vert \Gamma_i(x) \vert$ for all $x \in X$. We have $k_0=1$ and $k_1=k$. By \cite[p.~247]{bbit},
 \begin{align*}
 k_i = \frac{b_0 b_1 \cdots b_{i-1}}{c_1 c_2 \cdots c_i} \qquad \qquad (0 \leq i \leq D).
 \end{align*}
 
\noindent We recall the adjacency matrix of $\Gamma$. 
 Define $A \in {\rm Mat}_X(\mathbb C)$ with $(x,y)$-entry
 \begin{align*}
 A_{x,y} = \begin{cases} 1 & \hbox{ if $\partial(x,y)=1$};\\
 0 & \hbox{ if $\partial(x,y) \not=1$}
 \end{cases} \qquad \qquad x,y \in X.
 \end{align*}
 \noindent We call $A$ the {\it adjacency matrix} of $\Gamma$. The matrix $A$ is real and symmetric, so it is diagonalizable.
 By the {\it eigenvalues} (resp. {\it primitive idempotents}) of $\Gamma$ we mean the eigenvalues (resp. primitive idempotents) of $A$. By \cite[p.~128]{bcn}, $\Gamma $ has $D+1$ eigenvalues
 $k=\theta_0 > \theta_1 > \cdots > \theta_D$.
 For $0 \leq i \leq D$ let $E_i$ denote the primitive idempotent of $\Gamma$ associated with $\theta_i$.
 By \cite[Section~2]{int} the matrices $\lbrace E_i \rbrace_{i=0}^D$ form a basis for the subalgebra $M$ of ${\rm Mat}_X(\mathbb C)$ generated by $A$.
 The subalgebra $M$ is called the  {\it Bose-Mesner algebra} of $\Gamma$.
 \medskip
 
\noindent We recall the dual Bose-Mesner algebras of $\Gamma$. For the rest of this section, fix $x \in X$.
For $0 \leq i \leq D$ define a diagonal matrix $E^*_i = E^*_i (x)$ in ${\rm Mat}_X(\mathbb C)$ with $(y,y)$-entry
 \begin{align*}
 (E^*_i)_{y,y} = \begin{cases} 1 & \hbox{ if $\partial(x,y)=i$};\\
 0 & \hbox{ if $\partial(x,y) \not=i$}
 \end{cases} \qquad \qquad y \in X.
 \end{align*}
 \noindent
We call $E_i^*$ the  $i$th {\it dual primitive idempotent of $\Gamma$
 with respect to $x$} \cite[p.~378]{tSub1}. By construction
(i) $I=\sum_{i=0}^D E_i^*$;
(ii) $(E_i^*)^t = E_i^*$ $(0 \leq i \leq D)$;
(iii) $\overline{E^*_i} = E_i^*$ $(0 \leq i \leq D)$;
(iv) $E_i^*E_j^* = \delta_{i,j}E_i^* $ $(0 \leq i,j \leq D)$.
Therefore, the matrices
$\lbrace E_i^*\rbrace_{i=0}^D$ form a 
basis for a commutative subalgebra
$M^*=M^*(x)$ of 
${\rm Mat}_X(\mathbb C)$.
We call 
$M^*$ the {\it dual Bose-Mesner algebra of
$\Gamma$ with respect to $x$} \cite[p.~378]{tSub1}.
By construction,
\begin{align*}
E^*_i V = {\rm Span} \lbrace \hat y \vert y \in \Gamma_i(x) \rbrace \qquad \qquad (0 \leq i \leq D).
\end{align*}
\noindent Moreover,
\begin{align*}
V = \sum_{i=0}^D E^*_i V \qquad \qquad \hbox{(orthogonal direct sum).}
 \end{align*}
 \noindent By the triangle inequality,
 \begin{align}
 \label{eq:triangle}
 E^*_i A E^*_j  =0 \quad \hbox{\rm if} \; \vert i - j \vert > 1 \qquad \qquad (0 \leq i,j\leq D).
 \end{align}

 \noindent We recall the subconstituent algebra of $\Gamma$ with respect to $x$. Let $T=T(x)$ denote the subalgebra of ${\rm Mat}_X(\mathbb C)$ generated by $M$ and $M^*$.
 The algebra $T$ is finite-dimensional and noncommutative.  We call $T$ the {\it subconstituent algebra} (or {\it Terwilliger algebra}) of $\Gamma$ with respect to $x$, see \cite{tSub1, tSub2, tSub3}.
 \medskip
 
 \noindent We describe some elements in $T$.  Define the matrices
 \begin{align}
 L = \sum_{i=1}^D E^*_{i-1} A E^*_i, \qquad \quad F = \sum_{i=0}^D E^*_i A E^*_i, \qquad \quad R= \sum_{i=0}^{D-1} E^*_{i+1} A E^*_i.
 \label{eq:LFR}
\end{align}
By  $I = \sum_{i=0}^D E^*_i$ and \eqref{eq:triangle},
\begin{align*}
A = L + F + R.
\end{align*}
By construction,
\begin{align*}
&L E^*_i V \subseteq E^*_{i-1} V \qquad \quad (1 \leq i \leq D), \qquad L E^*_0V =0, \\
& F E^*_i V \subseteq E^*_i V \qquad \quad (0 \leq i \leq D), \\
&R E^*_i V \subseteq E^*_{i+1} V \qquad \quad (0 \leq i \leq D-1), \qquad RE^*_D V=0.
\end{align*}
\noindent We call $L$  (resp. $F$) (resp. $R$) the {\it lowering map} (resp. {\it flat map}) (resp. {\it raising map}) of $\Gamma$ with respect to $x$.
\medskip

 \noindent We recall the irreducible $T$-modules. The algebra $T$ is closed under the conjugate-transpose map, so $T$ is semisimple \cite[Lemma~3.4(i)]{tSub1}.
 By a {\it $T$-module} we mean a subspace $W \subseteq V$ such that $T W \subseteq W$. A $T$-module $W$ is {\it irreducible} whenever 
 $W \not=0$ and $W$ does not contain a $T$-module besides $0$ and $W$.
 For a $T$-module $W$ the orthogonal complement
 $W^\perp = \lbrace v \in V \vert \langle w,v \rangle =0\rbrace $ is a $T$-module. It follows that any $T$-module  is an orthogonal direct sum of
 irreducible $T$-modules. In particular, the standard module $V$ is an orthogonal direct sum of irreducible $T$-modules \cite[Lemma~3.4(ii)]{tSub1}.
 Let $W$ denote an irreducible $T$-module. Then $W$ is an orthogonal direct sum of the nonzero subspaces among $\lbrace E^*_iW\rbrace_{i=0}^D$ \cite[Lemma~3.4(iii)]{tSub1}.
 By the {\it endpoint} of $W$ we mean
 \begin{align*}
 {\rm min} \lbrace i \vert 0 \leq i \leq D, E^*_iW \not=0\rbrace.
 \end{align*}
 By the {\it diameter} of $W$ we mean
 \begin{align*}
\bigl \vert \lbrace i \vert 0 \leq i \leq D, E^*_i W \not=0\rbrace \bigr\vert -1.
 \end{align*}
 The $T$-module $W$ is called {\it thin} whenever 
 \begin{align*}
 {\rm dim}\, E^*_i W \leq 1 \qquad \qquad (0 \leq i \leq D).
 \end{align*}
 
 \noindent Let $W, W'$ denote irreducible $T$-modules. By a  {\it $T$-module isomorphism from $W$ to $W'$} we mean a $\mathbb C$-linear bijection $\sigma: W \to W'$
 such that $ B \sigma = \sigma B $ on $W$ for all $B \in T$. The $T$-modules $W, W'$ are said to be {\it isomorphic} whenever there exists a $T$-module isomorphism from
 $W$ to $W'$.
 Nonisomorphic irreducible $T$-modules are orthogonal  \cite[Lemma~3.3]{curtin}.
 \medskip
 
 \noindent There exists a unique irreducible $T$-module with endpoint 0, said to be {\it primary}. The primary $T$-module is described in \cite[Section~7]{int}.
 \medskip
 
 \noindent Let $W$ denote an irreducible $T$-module with endpoint 1 and ${\rm dim}\,E^*_1W=1$. By construction $F E^*_1 W \subseteq E^*_1W$. Therefore
 there exists $\eta \in \mathbb C$ such that $(F-\eta I) E^*_1 W=0$. We call $\eta$ the {\it local eigenvalue of $W$}.
 \medskip
 
 \noindent We will be discussing automorphisms of $\Gamma$. The concept of a graph
 automorphism is explained in \cite[p.~435]{bcn}.

 \section{The bilinear forms graph $\Gamma$}
 
 We turn our attention to the following graph $\Gamma=(X, \mathcal R)$. Fix a finite field ${\rm GF}(q)$ with $q\not=2$. Fix
 integers $D, N$ such that $N>2D\geq 6$.
 The vertex set $X$ consists of the $D \times (N-D)$ matrices that have all entries in ${\rm GF}(q)$. 
 Thus $\vert X \vert = q^{D(N-D)}$.
 Two vertices
 are adjacent whenever their difference has rank 1. The graph $\Gamma$ is called the bilinear forms graph and often denoted $H_q(D, N-D)$, see \cite[p.~280]{bcn}.
 The following facts about $\Gamma$ are taken from \cite[Section~9.5]{bcn}.
 The graph $\Gamma$ is distance-regular with diameter $D$ and intersection numbers
 \begin{align}
 c_i = q^{i-1} \frac{q^i -1}{q-1}, \qquad \quad b_i = \frac{(q^{N-D}-q^i)(q^{D}-q^i)}{q-1} 
 \qquad \qquad (0 \leq i \leq D). \label{eq:cibi}
 \end{align} 
 We have
 \begin{align}
 a_i =      \frac{q^i -1}{q-1} \Bigl           (q^{N-D}  +q^D   -q^i - q^{i-1}-1\Bigr) \qquad \qquad (0 \leq i \leq D). \label{eq:ai}
 \end{align}
 In particular,
 \begin{align}
 a_1 = q^{N-D} +q^D    -q-2, \qquad \quad 
 a_D = \frac{q^D-1}{q-1}             \Bigl(q^{N-D} - q^{D-1}-1\Bigr). \label{eq:a1}
 \end{align} 
 The eigenvalues of  $\Gamma$ are
 \begin{align}
 \theta_i = \frac{q^{N-i} + 1 - q^D-q^{N-D}}{q-1} \qquad \qquad (0 \leq i \leq D). \label{eq:theta}
 \end{align}
  
\noindent 
The graph $\Gamma$ is distance-transitive in the sense of  \cite[p.~136]{bcn},
and  formally self-dual in the sense of \cite[p.~49]{bcn}.
By \cite[p.~280]{bcn} we have
 \begin{align}
 \partial(x,y) = {\rm rank}\,(x-y) \qquad \qquad (x, y \in X).  \label{eq:Rdist}
 \end{align}
 
 \noindent More information about the bilinear forms graph can be found in \cite{delsarte, qtet,kim1,tanaka, alt}.
 
\section{The local graph with respect to a vertex  in $\Gamma$}
We continue to discuss the bilinear forms graph $\Gamma=(X, \mathcal R)$.
Throughout this section, we fix a vertex $x \in X$. Recall that $\Gamma(x)$ is the set of
vertices in $X$ that are adjacent to $x$. We have
\begin{align}
\vert \Gamma(x) \vert =k=  \frac{(q^{N-D}-1)(q^D-1)}{q-1}.       \label{eq:valency}
\end{align}
 The vertex subgraph of $\Gamma$ induced on $\Gamma(x)$
is called the {\it local graph of $\Gamma$ with respect to $x$}. For notational convenience, we denote this local graph by $\Gamma(x)$.
\medskip

\noindent By construction, the local graph $\Gamma(x)$ is regular with valency $a_1$.
To describe this graph in more detail, we use the notation
 \begin{align*}
 \lbrack n \rbrack = \frac{q^n-1}{q-1} \qquad \qquad n \in \mathbb N.
 \end{align*}
 By \cite[Section~6.1]{hobart}  the local graph $\Gamma(x)$ is a $(q-1)$-clique extension of a Cartesian product of complete graphs
 $K_{\lbrack D \rbrack} \times K_{\lbrack N-D\rbrack}$. In particular, the local graph $\Gamma(x)$ is connected.

\begin{lemma} \label{lem:localSpec} {\rm (See \cite[Theorem~6.1]{hobart}.)} 
For the local graph $\Gamma(x)$ the eigenvalues and
their multiplicities are given in the table below:
\begin{align*} 
\begin{tabular}[t]{c|c}
{\rm eigenvalue }& {\rm multiplicity}
 \\
 \hline
$a_1 = q^{N-D}+q^D-q-2$ & $1$ \\
$q^{N-D}-q-1$  &$ \frac{q^D-q}{q-1}$ \\
$q^D-q-1$ & $\frac{q^{N-D}-q}{q-1}$ \\
$-1$ & $ \frac{(q^D-1)(q^{N-D}-1)(q-2)}{(q-1)^2}$ \\
$-q$ & $\frac{(q^D-q)(q^{N-D}-q)}{(q-1)^2}$
    \end{tabular}
\end{align*}
\end{lemma}

\section{The $(x,y)$-partition of $X$}
We continue to discuss the bilinear forms graph $\Gamma=(X, \mathcal R)$. For the rest of this paper, we fix adjacent vertices $x,y \in X$.
We will use $x,y$ to get an equitable partition of $X$, called the $(x,y)$-partition. We will proceed as follows.
In this section, we describe the $(x,y)$-partition of $X$ from three points of view. In the next section, we will show that the $(x,y)$-partition of $X$ is equitable.
\medskip

\noindent We now describe the $(x,y)$-partition of $X$ from the first point of view.

\begin{proposition} \label{prop:parts}
There exists a partition of $X$ into $6D-2$ nonempty subsets
\begin{align*}
&O_{i,i-1} \quad (1 \leq i \leq D), \qquad \qquad \quad O^A_{i,i} \quad (2 \leq i \leq D), \\
&O^B_{i,i} \quad (1 \leq i \leq D), \qquad \qquad \qquad O^C_{i,i} \quad (1 \leq i \leq D), \\
&O^D_{i,i} \quad (1 \leq i \leq D-1), \qquad \qquad O_{i-1,i} \quad (1 \leq i \leq D)
\end{align*}
such that the following {\rm (i)--\rm (vi)}  holds.
\begin{enumerate}
\item[\rm (i)] 
For $1 \leq i \leq D$,
\begin{align*}
O_{i,i-1}= \Gamma_i(x) \cap \Gamma_{i-1}(y).
\end{align*}
\item[\rm (ii)] For $2 \leq i \leq D$, the set $O^A_{i,i}$ consists of the vertices $z \in \Gamma_i(x) \cap \Gamma_i(y)$ such that 
\begin{align*}
\vert \Gamma(x) \cap \Gamma(y) \cap \Gamma_{i-1}(z) \vert = 2 q^{i-1}, \qquad \qquad 
\vert \Gamma(x) \cap \Gamma(y) \cap \Gamma_{i+1}(z) \vert = 0.
 \end{align*}
\item[\rm (iii)] For $1 \leq i \leq D$, the set $O^B_{i,i}$ consists of the vertices $z \in  \Gamma_i(x) \cap \Gamma_i(y)$ such that 
\begin{align*}
\vert \Gamma(x) \cap \Gamma(y) \cap \Gamma_{i-1}(z) \vert = 2 q^{i-1}-1, \qquad \qquad 
\vert \Gamma(x) \cap \Gamma(y) \cap \Gamma_{i+1}(z) \vert = 0.
\end{align*}
\item[\rm (iv)] For $1 \leq i \leq D$, the set  $O^C_{i,i}$ consists of the vertices $z \in  \Gamma_i(x) \cap \Gamma_i(y)$ such that 
\begin{align*}
\vert \Gamma(x) \cap \Gamma(y) \cap \Gamma_{i-1}(z) \vert = q^{i-1}, \qquad \qquad 
\vert \Gamma(x) \cap \Gamma(y) \cap \Gamma_{i+1}(z) \vert = q^D-q^i.
\end{align*}
\item[\rm (v)] For $1 \leq i \leq D-1$, the set $O^D_{i,i}$ consists of the vertices $z \in  \Gamma_i(x) \cap \Gamma_i(y)$ such that 
\begin{align*}
\vert \Gamma(x) \cap \Gamma(y) \cap \Gamma_{i-1}(z) \vert = q^{i-1}, \qquad \qquad 
\vert \Gamma(x) \cap \Gamma(y) \cap \Gamma_{i+1}(z) \vert = q^{N-D}-q^i.
\end{align*}
\item[\rm (vi)] 
For $1 \leq i \leq D$,
\begin{align*}
O_{i-1,i}= \Gamma_{i-1}(x) \cap \Gamma_{i}(y).
\end{align*}
\end{enumerate}
\end{proposition}

\noindent The proof of Proposition \ref{prop:parts} will be completed later in this section.
\medskip

\noindent Referring   to Proposition \ref{prop:parts}, if we swap $x \leftrightarrow y$ then $O_{i,i-1}$ and $O_{i-1,i}$ get swapped
and each of
 $O^A_{i,i}$, $O^B_{i,i}$,  $O^C_{i,i}$, $O^D_{i,i}$ is unchanged.

\begin{definition}\label{def:xyPartition} \rm
The partition of $X$ from Proposition \ref{prop:parts}  is called the {\it $(x,y)$-partition of $X$}.
\end{definition}

\noindent We just described the $(x,y)$-partition of $X$ from the first point of view. We will give the second point of view after a brief
discussion.
\medskip

\noindent Recall that the vertex set $X$ consists of the $D \times (N-D)$ matrices that
have all entries in ${\rm GF}(q)$. For some calculations, it is convenient to invoke
distance-transitivity  and make the following assumption.
\begin{assumption} \label{assume} \rm Under the {\it standard assumption},
\begin{enumerate}
\item[\rm (i)] $x=0$;
\item[\rm (ii)] $y$ has $(1,1)$-entry $1$ and all other entries $0$.
\end{enumerate}
\end{assumption}

\noindent  For the rest of this section, the standard assumption  is in force.
\medskip

\noindent We describe some automorphisms of $\Gamma$.  The general linear groups
\begin{align*}
{\rm GL}_D\bigl({\rm GF}(q) \bigr), \qquad \qquad {\rm GL}_{N-D}\bigl({\rm GF}(q) \bigr)
 \end{align*}
  act  on $X$ by left and right multiplication, respectively. These actions  commute, so they induce a group homomorphism
from the direct sum ${\rm GL}_D\bigl({\rm GF}(q) \bigr)\times {\rm GL}_{N-D}\bigl({\rm GF}(q) \bigr) $ into the automorphism group ${\rm Aut}(\Gamma)$.
The automorphisms from ${\rm GL}_D\bigl({\rm GF}(q) \bigr)\times {\rm GL}_{N-D}\bigl({\rm GF}(q) \bigr) $  fix $x$, but do not fix $y$ in general.
\medskip

\noindent  Next, we describe some automorphisms 
from  ${\rm GL}_D\bigl({\rm GF}(q) \bigr)\times {\rm GL}_{N-D}\bigl({\rm GF}(q) \bigr) $ that fix both $x$ and $y$. 
There are some elements in ${\rm GL}_D\bigl({\rm GF}(q) \bigr)$  called elementary row operations.
 The following elementary row operations fix $y$ as well as $x$:
\begin{enumerate}
\item[\rm (i)] interchange row $i$ and row $j$  $(2 \leq i<j\leq D)$;
\item[\rm (ii)] multiply row $i$ by a nonzero scalar $(2 \leq i \leq D)$;
\item[\rm (iii)] replace row $i$ by row $i$ plus a scalar multiple of row $j$ $(1 \leq i \leq D)$, $(2 \leq j \leq D)$.
\end{enumerate}
The elementary row operations of the
above types (i)--(iii) will be called {\it allowed}.
Let $G_{\rm row}$ denote the subgroup of  ${\rm GL}_D\bigl({\rm GF}(q) \bigr)$
 generated by the allowed elementary row operations.
  By construction, every element of $G_{\rm row}$ fixes each of $x,y$.
\medskip

\noindent
There are some elements in ${\rm GL}_{N-D}\bigl({\rm GF}(q) \bigr)$ called elementary column operations.
 The following elementary column operations fix $y$ as well as $x$:
\begin{enumerate}
\item[\rm (i)] interchange column $i$ and column $j$  $(2 \leq i<j\leq N-D)$;
\item[\rm (ii)] multiply column $i$ by a nonzero scalar $(2 \leq i \leq N-D)$;
\item[\rm (iii)] replace column $i$ by column $i$ plus a scalar multiple of column $j$ $(1 \leq i \leq N-D)$, $(2 \leq j \leq N-D)$.
\end{enumerate}
The elementary column operations of the
above types (i)--(iii) will be called {\it allowed}.
Let $G_{\rm col}$ denote the subgroup of  ${\rm GL}_{N-D}\bigl({\rm GF}(q) \bigr)$
 generated by the allowed elementary column operations.
  By construction, every element of $G_{\rm col}$ fixes each of $x,y$.
\medskip

\noindent Let $G$ denote the subgroup of
  ${\rm GL}_D\bigl({\rm GF}(q) \bigr)\times {\rm GL}_{N-D}\bigl({\rm GF}(q) \bigr) $ 
  generated by
$G_{\rm row}$ and $G_{\rm col}$. The group $G$ is isomorphic to the direct sum $G_{\rm row} \times G_{\rm col}$.
By construction, every element of $G$ fixes each of $x,y$. By construction, for $z \in X$ the following are the same:
\begin{enumerate}
\item[\rm (i)] the $G$-orbit containing $z$;
\item[\rm (ii)] the set $G_{\rm row} z G_{\rm col}$.
\end{enumerate}

\noindent
In the next result we display some matrices in $X$.
For each displayed matrix we  show the nonzero part, and sometimes insert an extra 0 to clarify the position of the entries. 
\medskip

\noindent We now describe the $(x,y)$-partition of $X$ from a second point of view.

\begin{proposition} \label{prop:part2} \rm The $(x,y)$-partition of $X$ 
 looks as follows under the standard assumption.
\begin{enumerate}
\item[\rm (i)] For $1 \leq i \leq D$,
\begin{align*}
&O_{i,i-1} = G_{\rm row}  \begin{pmatrix} I_{i} & & \\
&  & \\
\end{pmatrix} G_{\rm col}.
\end{align*}
\item[\rm (ii)] For $2 \leq i \leq D$,
\begin{align*}
 O_{i,i}^A  = G_{\rm row}  \begin{pmatrix} 0&&1\\
&I_{i-2} &  \\
1&  & \\
\end{pmatrix} G_{\rm col}.
\end{align*}
\item[\rm (iii)] For $1 \leq i \leq D$,
\begin{align*}
&O_{i,i}^B  =  \bigcup_{\xi \not=0,1}    \Biggl\lbrace  G_{\rm row} \begin{pmatrix} \xi &&\\
&I_{i-1} &  \\
&  & \\
\end{pmatrix} G_{\rm col} \Biggr \rbrace.
\end{align*}
\item[\rm (iv)] For $1 \leq i \leq D$,
\begin{align*}
O_{i,i}^C = G_{\rm row} \begin{pmatrix} 0&&1\\
&I_{i-1} &  \\
&  & \\
\end{pmatrix} G_{\rm col}.
\end{align*}
\item[\rm (v)] For $1 \leq i \leq D-1$,
\begin{align*}
O_{i,i}^D =G_{\rm row} \begin{pmatrix} 0&&\\
&I_{i-1} &  \\
1&  & \\
\end{pmatrix} G_{\rm col}.
\end{align*}
\item[\rm (vi)] For $1 \leq i \leq D$,
\begin{align*}
O_{i-1,i} = G_{\rm row} \begin{pmatrix} 0&&\\
&I_{i-1} &  \\
&  & \\
\end{pmatrix} G_{\rm col}.
\end{align*}
\end{enumerate}
\end{proposition}

\noindent The proof of Proposition \ref{prop:part2} will be completed later in this section.
\medskip

\noindent  We just described the $(x,y)$-partition of $X$ from the second point of view. We will give the third point of view
after a brief discussion.

\begin{definition}\label{def:zVar} \rm For $z \in X$,
\begin{enumerate}
\item[\rm (i)] 
 let $\vert z$ denote the matrix obtained by replacing each entry in the first column  of $z$ by $0$;
 \item[\rm (ii)] let $\overline z$ denote the matrix obtained by replacing each entry in the first row of $z$ by $0$;
 \item[\rm (iii)] let $\ulcorner z$ denote the matrix obtained by replacing each entry in the first column  and first row  of $z$ by $0$.
\end{enumerate}
\end{definition}
\begin{example} \rm For $D=3$ and $N=7$ an example is
\begin{align*}
& z = \begin{pmatrix} 1&1&1&1 \\
                               1&1&1&1 \\
                               1&1&1&1
                               \end{pmatrix}, 
                               \qquad \qquad
                              \vert  z = \begin{pmatrix} 0&1&1&1 \\
                               0&1&1&1 \\
                               0&1&1&1
                               \end{pmatrix}, \\
        & \overline z = \begin{pmatrix} 0&0&0&0 \\
                               1&1&1&1 \\
                               1&1&1&1
                               \end{pmatrix},   \qquad \qquad    \ulcorner z = \begin{pmatrix} 0&0&0&0 \\
                               0&1&1&1 \\
                               0&1&1&1
                               \end{pmatrix}.                                  
\end{align*}
\end{example}

\begin{lemma} For $z \in X$ we have
\begin{align*}
\vert (z-y) = \vert z, \qquad \qquad            \overline {z-y} = \overline z,  \qquad \qquad \ulcorner (z-y) = \ulcorner z.
\end{align*}
\end{lemma}
\begin{proof} The matrices $z-y$ and $z$ agree at every entry except the $(1,1)$-entry.
\end{proof}

\noindent For $z \in X$ we consider the parameters
\begin{align}
{\rm rank}(z), \qquad 
{\rm rank}(z-y), \qquad 
{\rm rank}(\vert z), \qquad 
{\rm rank}(\overline z), \qquad 
{\rm rank}(\ulcorner z).
\label{eq:param}
\end{align}

\noindent We have some comments.

\begin{lemma} \label{lem:dist} For $z \in X$ we have
\begin{align*}
\partial(x,z) = {\rm rank}(z), \qquad \qquad \partial(y,z) = {\rm rank}(z-y).
\end{align*}
\end{lemma}
\begin{proof} By \eqref{eq:Rdist} and the standard assumption.
\end{proof}

\begin{lemma} \label{lem:Gfix} For $z \in X$ and $g \in G$, each parameter in \eqref{eq:param}  is unchanged if we replace $z$ by $g(z)$.
\end{lemma}
\begin{proof} By linear algebra.
\end{proof}

\noindent We now describe the $(x,y)$-partition of $X$ from a third point of view.

\begin{proposition} \label{prop:part3} The $(x,y)$-partition of $X$ looks as follows 
under the standard assumption.

\begin{enumerate}
\item[\rm (i)]
For $1 \leq i \leq D$, the set $O_{i,i-1}$ consists of the vertices $z \in X$ such that
\begin{align*}
&{\rm rank}(z) = i, \qquad \qquad {\rm rank}(z-y)=i-1, \\
&
{\rm rank}(\vert z) = i-1, \qquad 
{\rm rank}(\overline z) = i-1, \qquad 
{\rm rank}(\ulcorner z) = i-1.
\end{align*}
\item[\rm (ii)] For $2 \leq i \leq D$, the set $O^A_{i,i}$ consists of the vertices $z \in X$ such that 
\begin{align*}
&{\rm rank}(z)=i, \qquad \qquad {\rm rank}(z-y) =i, \\
&{\rm rank}( \vert z)=i-1, \qquad  {\rm rank}(\overline z) = i-1, \qquad   {\rm rank}(\ulcorner z) = i-2.
 \end{align*}
\item[\rm (iii)] For $1 \leq i \leq D$, the set $O^B_{i,i}$ consists of the vertices $z \in X$ such that 
\begin{align*}
&{\rm rank}(z)=i, \qquad \qquad {\rm rank}(z-y)=i, \\
&{\rm rank}(\vert z) = i-1, \qquad  {\rm rank}(\overline z) = i-1, \qquad  {\rm rank}(\ulcorner z) = i-1.
\end{align*}
\item[\rm (iv)] For $1 \leq i \leq D$, the set $O^C_{i,i}$ consists of the vertices $z \in X$ such that 
\begin{align*}
&{\rm rank}(z)=i, \qquad \qquad {\rm rank}(z-y)=i,\\
&{\rm rank}(\vert z) = i, \qquad {\rm rank}(\overline z) = i-1, \qquad  {\rm rank}(\ulcorner z) = i-1.
\end{align*}
\item[\rm (v)] For $1 \leq i \leq D-1$, the set $O^D_{i,i}$ consists of the vertices $z \in X$ such that
\begin{align*}
&{\rm rank}(z)=i,\qquad \qquad {\rm rank}(z-y)=i, \\
&{\rm rank}(\vert z) = i-1, \qquad  {\rm rank}(\overline z)=i, \qquad               {\rm rank}(\ulcorner z) = i-1.
\end{align*}
\item[\rm (vi)]
For $1 \leq i \leq D$, the set  $O_{i-1,i}$ consists of the vertices $z \in X$ such that
\begin{align*}
&{\rm rank}(z) = i-1, \qquad \qquad {\rm rank}(z-y) = i, \\
&
{\rm rank}(\vert z) = i-1, \qquad 
{\rm rank}(\overline z) = i-1, \qquad 
{\rm rank}(\ulcorner z) = i-1.
\end{align*}
\end{enumerate}
\end{proposition}

\noindent {\it Proof of Propositions \ref{prop:parts}, \ref{prop:part2}, \ref{prop:part3}}. Without loss of generality, 
we make the standard assumption. We define the subsets
\begin{align}
&O_{i,i-1} \quad (1 \leq i \leq D), \qquad \qquad \quad O^A_{i,i} \quad (2 \leq i \leq D), \label{eq:p1}\\
&O^B_{i,i} \quad (1 \leq i \leq D), \qquad \qquad \qquad O^C_{i,i} \quad (1 \leq i \leq D), \label{eq:p2} \\
&O^D_{i,i} \quad (1 \leq i \leq D-1), \qquad \qquad O_{i-1,i} \quad (1 \leq i \leq D) \label{eq:p3}
\end{align}
to match the descriptions in Proposition \ref{prop:part2}. By construction, each of these subsets is nonempty.
Also, their union  is equal to $X$ because via the allowed elementary
row and column operations we can  transform any vertex in $X$  into at least one of the forms shown in Proposition \ref{prop:part2}.
We define the subsets
\begin{align*}
&{\mathcal O}_{i,i-1} \quad (1 \leq i \leq D), \qquad \qquad \quad {\mathcal O}^A_{i,i} \quad (2 \leq i \leq D), \\
&{\mathcal O}^B_{i,i} \quad (1 \leq i \leq D), \qquad \qquad \qquad {\mathcal O}^C_{i,i} \quad (1 \leq i \leq D), \\
&{\mathcal O}^D_{i,i} \quad (1 \leq i \leq D-1), \qquad \qquad {\mathcal O}_{i-1,i} \quad (1 \leq i \leq D)
\end{align*}
to match the descriptions in Proposition \ref{prop:part3}.  By construction, these subsets are mutually disjoint.
By Definition \ref{def:zVar} and Lemma \ref{lem:Gfix},
\begin{align}
&O_{i,i-1}\subseteq {\mathcal O}_{i,i-1} \quad (1 \leq i \leq D), \qquad \qquad \quad O^A_{i,i}\subseteq {\mathcal O}^A_{i,i} \quad (2 \leq i \leq D), \label{eq:q1}\\
&O^B_{i,i}\subseteq {\mathcal O}^B_{i,i} \quad (1 \leq i \leq D), \qquad \qquad \qquad O^C_{i,i} \subseteq {\mathcal O}^C_{i,i} \quad (1 \leq i \leq D), \label{eq:q2}\\
&O^D_{i,i}\subseteq {\mathcal O}^D_{i,i} \quad (1 \leq i \leq D-1), \qquad \qquad  O_{i-1,i} \subseteq {\mathcal O}_{i-1,i} \quad (1 \leq i \leq D) \label{eq:q3}.
\end{align}
By these comments, $X$ is partitioned into the sets shown in \eqref{eq:p1}--\eqref{eq:p3}, and equality holds everywhere in \eqref{eq:q1}--\eqref{eq:q3}.
By routine combinatorial counting, we find that the  sets shown in \eqref{eq:p1}--\eqref{eq:p3} fit the descriptions in Proposition
\ref{prop:parts}. The results follow.
\hfill $\Box$
\medskip

\noindent We emphasize a  feature of the $(x,y)$-partition of $X$.

\begin{lemma} \label{lem:disjUnion} For the $(x,y)$-partition of $X$,
\begin{enumerate}
\item[\rm (i)] $\Gamma(x) \cap \Gamma(y)$ is a disjoint union of $O^B_{1,1}$, $O^C_{1,1}$, $O^D_{1,1}$;
\item[\rm (ii)] for $2 \leq i \leq D-1$ the set $\Gamma_i(x) \cap \Gamma_i(y)$ is a disjoint union of $O^A_{i,i}$, $O^B_{i,i}$, $O^C_{i,i}$, $O^D_{i,i}$;
\item[\rm (iii)] $\Gamma_D(x) \cap \Gamma_D(y)$ is a disjoint union of $O^A_{D,D}$, $O^B_{D,D}$, $O^C_{D,D}$.
\end{enumerate}
\end{lemma}
\begin{proof} By Proposition \ref{prop:parts}.
\end{proof}

\section{The $(x,y)$-partition of $X$ is equitable}

We continue to discuss the bilinear forms graph $\Gamma=(X, \mathcal R)$ and the adjacent vertices $x,y \in X$.
Recall the $(x,y)$-partition of $X$ from Definition \ref{def:xyPartition}.
In this section, we show that the $(x,y)$-partition of $X$ is equitable in the sense of \cite[p.~436]{bcn}.

\begin{lemma} \label{lem:bubbleSize}
We have
\begin{align*}
&\vert O_{i,i-1} \vert = \frac{k_i c_i}{k} \quad (1 \leq i \leq D), \qquad \qquad \vert O^A_{i,i} \vert = \frac{k_i}{k} \,\frac{(q^i-1)(q^{i-1}-1)}{q-1} \quad (2 \leq i \leq D), \\
& \vert O^B_{i,i} \vert = \frac{k_i}{k} \,\frac{q^{i-1}(q^i-1)(q-2)}{q-1}, \qquad \quad 
 \vert O^C_{i,i} \vert = \frac{k_i}{k} \,\frac{(q^{N-D}-q^i)(q^i-1)}{q-1} \qquad (1 \leq i \leq D), \\
 &  \vert O^D_{i,i} \vert = \frac{k_i}{k} \,\frac{(q^{D}-q^i)(q^i-1)}{q-1} \quad (1 \leq i \leq D-1), 
 \qquad \qquad \vert O_{i-1,i} \vert = \frac{k_i c_i}{k} \quad (1 \leq i \leq D).
\end{align*}
\end{lemma}
\begin{proof}
By combinatorial counting.
\end{proof}

\begin{lemma} \label{lem:structure1}
Let $1\leq i \leq D$. For each set in the tables below, we give the number of vertices in the set that are adjacent to a given vertex in $O_{i,i-1}$:
\bigskip

\begin{tabular}{c|c|c}

$O_{i-1,i-2}$ & $O_{i-1,i}$ & $O_{i,i-1}$  \\          
\hline
$q^{i-2} \frac{q^{i-1}-1}{q-1}$& $q^{2i-2}-q^{2i-3}+q^{i-2}$&$\frac{q^{i-1}-1}{q-1}\bigl(q^{N-D}+q^D-3q^{i-1}+q^{i-2}-1\bigr)$    \\     

\end{tabular}

\bigskip

\begin{tabular}{c|c|c|c|c}

 $O_{i+1,i}$ &$O_{i-1,i-1}^A$ & $O_{i-1,i-1}^B$ & $O_{i-1,i-1}^C$ & $O_{i-1,i-1}^D$ \\
\hline
  $\frac{(q^{N-D}-q^i)(q^D-q^i)}{q-1}$            &   $0$& $0$& $q^{i-2}(q^{i-1}-1)$& $q^{i-2}(q^{i-1}-1)$\\

\end{tabular}

\bigskip

\begin{tabular}{c|c|c|c}

$O_{i,i}^A$ & $O_{i,i}^B$ & $O_{i,i}^C$ & $O_{i,i}^D$ \\
\hline
$(q^{i-1}-1)(q^{i-1}-q^{i-2})$& $q^{i-2}(q-2)(q^i-q^{i-1}+1)$& $q^{i-1}(q^{N-D}-q^i)$& $q^{i-1}(q^{D}-q^i)$\\

\end{tabular}
\end{lemma}
\begin{proof} By combinatorial counting.
\end{proof}

\begin{lemma} \label{lem:structureA}
Let $2\leq i \leq D$. For each set in the tables below, we give the number of vertices in the set that are adjacent to a given vertex in $O^A_{i,i}$:
\bigskip

\begin{tabular}{c|c|c|c}

$O_{i-1,i}$ & $O_{i,i-1}$ & $O_{i,i+1}$ & $O_{i+1,i}$ \\
\hline
$q^{i-1}(q^{i-1}-q^{i-2})$& $q^{i-1}(q^{i-1}-q^{i-2})$&$0$ &$0$ \\
\end{tabular}

\bigskip

\begin{tabular}{c|c|c|c}

$O_{i-1,i-1}^A$ & $O_{i-1,i-1}^B$ & $O_{i-1,i-1}^C$ & $O_{i-1,i-1}^D$ \\
\hline
$q^{i-1} \frac{ q^{i-2}-1}{q-1}$ & $0$ & $q^{2i-3} $ & $q^{2i-3} $  \\

\end{tabular}

\bigskip

\begin{tabular}{c|c}

$O_{i,i}^A$ & $O_{i,i}^B$ \\
\hline
$\frac{ q^{i-1}-1}{q-1}\bigl( q^{N-D}+q^D-3q^{i-1}+q^{i-2}-2 \bigr) + \frac{q^{i-2}-1}{q-1} $ & $q^{2i-3}(q-1)(q-2) $ \\
\end{tabular}

\bigskip

\begin{tabular}{c|c|c|c|c|c}

 $O_{i,i}^C$   &$O_{i,i}^D$ & $O_{i+1,i+1}^A$ & $O_{i+1,i+1}^B$ & $O_{i+1,i+1}^C$ & $O_{i+1,i+1}^D$ \\
\hline
 $q^{i-1}(q^{N-D}-q^i) $ &$ q^{i-1}(q^{D}-q^i)$ &$\frac{(q^D-q^i)(q^{N-D}-q^i)}{q-1} $ & $ 0$ & $ 0$ & $ 0$

\end{tabular}
\end{lemma}
\begin{proof} By combinatorial counting.
\end{proof}

\begin{lemma} \label{lem:structureB}
Let $1\leq i \leq D$. For each set in the tables below, we give the number of vertices in the set that are adjacent to a given vertex in $O^B_{i,i}$:
\bigskip

\begin{tabular}{c|c|c|c}

$O_{i-1,i}$ & $O_{i,i-1}$ & $O_{i,i+1}$ & $O_{i+1,i}$ \\
\hline
$q^{i-2}(q^i-q^{i-1}+1)$& $q^{i-2}(q^i-q^{i-1}+1)$ &$0$&$0$ \\
\end{tabular}

\bigskip

\begin{tabular}{c|c|c|c}

$O_{i-1,i-1}^A$ & $O_{i-1,i-1}^B$ & $O_{i-1,i-1}^C$ & $O_{i-1,i-1}^D$ \\
\hline
$0$ & $q^{i-2} \frac{q^{i-1}-1}{q-1} $ & $q^{i-2}(q^{i-1}-1) $ & $q^{i-2}(q^{i-1}-1)$  \\

\end{tabular}

\bigskip

\begin{tabular}{c|c|c|c}

$O_{i,i}^A$ & $O_{i,i}^B$ & $O_{i,i}^C$ & $O_{i,i}^D$ \\
\hline
$q^{i-2}(q-1)(q^{i-1}-1) $ & $\alpha_i$ & $q^{i-1}(q^{N-D}-q^i) $ & $ q^{i-1}(q^{D}-q^i)$  \\
\end{tabular}

\bigskip

\begin{tabular}{c|c|c|c}

$O_{i+1,i+1}^A$ & $O_{i+1,i+1}^B$ & $O_{i+1,i+1}^C$ & $O_{i+1,i+1}^D$ \\
\hline
$0 $ & $\frac{(q^D-q^i)(q^{N-D}-q^i)}{q-1} $ & $0 $ & $ 0$  \\

\end{tabular}

\bigskip

\bigskip

$\alpha_i = \frac{q^{i-1}-1}{q-1} \Bigl( q^{N-D}+q^D+q^{i+1}-5q^{i}+4q^{i-1}-2q^{i-2} +q^2-4q+2 \Bigr) +q-3. $
\end{lemma}
\begin{proof} By combinatorial counting.
\end{proof}

\begin{lemma} \label{lem:structureC}
Let $1\leq i \leq D$. For each set in the tables below, we give the number of vertices in the set that are adjacent to a given vertex in $O^C_{i,i}$:
\bigskip

\begin{tabular}{c|c|c|c}

$O_{i-1,i}$ & $O_{i,i-1}$ & $O_{i,i+1}$ & $O_{i+1,i}$ \\
\hline
$q^{2i-2}$&$q^{2i-2}$ & $q^{i-1}(q^D-q^i)$&$q^{i-1}(q^D-q^i)$ \\
\end{tabular}

\bigskip

\begin{tabular}{c|c|c|c}

$O_{i-1,i-1}^A$ & $O_{i-1,i-1}^B$ & $O_{i-1,i-1}^C$ & $O_{i-1,i-1}^D$ \\
\hline
$ 0$ & $0$ & $q^{i-1}\frac{q^{i-1}-1}{q-1}$ & $ 0$  \\

\end{tabular}

\bigskip

\begin{tabular}{c|c|c|c}

$O_{i,i}^A$ & $O_{i,i}^B$ & $O_{i,i}^C$ & $O_{i,i}^D$ \\
\hline
$ q^{i-1}(q^{i-1}-1)$ & $q^{2i-2}(q-2) $ & $\frac{q^i-1}{q-1}(q^{N-D}-q^i-1) + \frac{q^{i-1}-1}{q-1}(q^{D}-q^i) $ & $ 0$  \\

\end{tabular}

\bigskip

\begin{tabular}{c|c|c|c}

$O_{i+1,i+1}^A$ & $O_{i+1,i+1}^B$ & $O_{i+1,i+1}^C$ & $O_{i+1,i+1}^D$ \\
\hline
$q^{i-1} (q^{D}-q^i) $ & $q^{i-1}(q^{D}-q^i)(q-2) $ & $\frac{(q^{D}-q^i)(q^{N-D}-q^{i+1})}{q-1} $ & $ 0$  \\
\end{tabular}

\end{lemma}
\begin{proof} By combinatorial counting.
\end{proof}

\begin{lemma} \label{lem:structureD}
Let $1\leq i \leq D-1$. For each set in the tables below, we give the number of vertices in the set that are adjacent to a given vertex in $O^D_{i,i}$:
\bigskip

\begin{tabular}{c|c|c|c}

$O_{i-1,i}$ & $O_{i,i-1}$ & $O_{i,i+1}$ & $O_{i+1,i}$ \\
\hline
$q^{2i-2}$&$q^{2i-2}$ & $q^{i-1}(q^{N-D}-q^i)$&$q^{i-1}(q^{N-D}-q^i)$ \\
\end{tabular}

\bigskip

\begin{tabular}{c|c|c|c}

$O_{i-1,i-1}^A$ & $O_{i-1,i-1}^B$ & $O_{i-1,i-1}^C$ & $O_{i-1,i-1}^D$ \\
\hline
$ 0$ & $ 0$ & $0 $ & $q^{i-1}\frac{q^{i-1}-1}{q-1} $  \\

\end{tabular}

\bigskip

\begin{tabular}{c|c|c|c}

$O_{i,i}^A$ & $O_{i,i}^B$ & $O_{i,i}^C$ & $O_{i,i}^D$ \\
\hline
$ q^{i-1}(q^{i-1}-1)$ & $ q^{2i-2}(q-2)$ & $0 $ & $\frac{q^i-1}{q-1}(q^{D}-q^i-1) + \frac{q^{i-1}-1}{q-1}(q^{N-D}-q^i) $  \\

\end{tabular}

\bigskip

\begin{tabular}{c|c|c|c}

$O_{i+1,i+1}^A$ & $O_{i+1,i+1}^B$ & $O_{i+1,i+1}^C$ & $O_{i+1,i+1}^D$ \\
\hline
$ q^{i-1}(q^{N-D}-q^i) $ & $q^{i-1}(q^{N-D}-q^i)(q-2)$ & $ 0$ & $ \frac{(q^{D}-q^{i+1})(q^{N-D}-q^i)}{q-1} $  \\

\end{tabular}
\end{lemma}
\begin{proof} By combinatorial counting.
\end{proof}

\begin{lemma} \label{lem:structure6}
Let $1\leq i \leq D$. For each set in the tables below, we give the number of vertices in the set that are adjacent to a given vertex in $O_{i-1,i}$:
\bigskip

\begin{tabular}{c|c|c}

$O_{i-2,i-1}$ & $O_{i,i-1}$ & $O_{i-1,i}$ \\ 
\hline
$q^{i-2}\frac{q^{i-1}-1}{q-1}$& $q^{2i-2}-q^{2i-3}+q^{i-2}$&$\frac{q^{i-1}-1}{q-1}\bigl(q^{N-D}+q^D-3q^{i-1}+q^{i-2}-1\bigr)$ \\

\end{tabular}

\bigskip

\begin{tabular}{c|c|c|c|c}

$O_{i,i+1}$&$O_{i-1,i-1}^A$ & $O_{i-1,i-1}^B$ & $O_{i-1,i-1}^C$ & $O_{i-1,i-1}^D$ \\
\hline
$\frac{(q^{N-D}-q^i)(q^D-q^i)}{q-1}$&$0$& $0$& $q^{i-2}(q^{i-1}-1)$& $q^{i-2}(q^{i-1}-1)$\\

\end{tabular}

\bigskip

\begin{tabular}{c|c|c|c}

$O_{i,i}^A$ & $O_{i,i}^B$ & $O_{i,i}^C$ & $O_{i,i}^D$ \\
\hline
$(q^{i-1}-1)(q^{i-1}-q^{i-2})$& $q^{i-2}(q-2)(q^i-q^{i-1}+1)$& $q^{i-1}(q^{N-D}-q^i)$& $q^{i-1}(q^{D}-q^i)$\\

\end{tabular}
\end{lemma}
\begin{proof} 
Similar to the proof of Lemma \ref{lem:structure1}.
\end{proof}

\begin{proposition} \label{prop:equit} The $(x,y)$-partition of $X$  is equitable.
\end{proposition}
\begin{proof} By Lemmas \ref{lem:structure1}--\ref{lem:structure6}.
\end{proof}

\newpage
\begin{landscape}
\vspace*{4cm}
\begin{tikzpicture}[
    nodeA/.style={circle,draw,minimum size=6mm,inner sep=0pt,font=\tiny},
    nodeB/.style={circle,draw,minimum size=3mm,inner sep=0pt,font=\tiny},
    x=1.3cm, y=1cm
]

\node[nodeA] (v0) at (0,0) {${}$};
\node[nodeA] (v5) at (-2,0) {$y$};
\node[nodeA] (v1) at (-2,5) {$x$};

\node[nodeB] (v2) at (-1,1.25) {$D$};
\node[nodeB] (v3) at (-1,2.5) {$B$};
\node[nodeB] (v4) at (-1,3.75) {$C$};

\node[nodeA] (v6) at (0,5) {${}$};

\node[nodeB] (v7) at (1,1) {$D$};
\node[nodeB] (v8) at (1,2) {$B$};
\node[nodeB] (v9) at (1,3) {$A$};
\node[nodeB] (v10) at (1,4) {$C$};

\node[nodeA] (v11) at (2,0) {${}$};
\node[nodeA] (v12) at (2,5) {${}$};

\node[nodeB] (v13) at (3,1) {$D$};
\node[nodeB] (v14) at (3,2) {$B$};
\node[nodeB] (v15) at (3,3) {$A$};
\node[nodeB] (v16) at (3,4) {$C$};


\node[nodeA] (v17) at (4,0) {${}$};
\node[nodeA] (v18) at (4,5) {${}$};

\node[nodeB] (v19) at (5,1) {$D$};
\node[nodeB] (v20) at (5,2) {$B$};
\node[nodeB] (v21) at (5,3) {$A$};
\node[nodeB] (v22) at (5,4) {$C$};

\node[nodeA] (v23) at (6,0) {${}$};
\node[nodeA] (v24) at (6,5) {${}$};

\node[nodeB] (v25) at (7,1) {$D$};
\node[nodeB] (v26) at (7,2) {$B$};
\node[nodeB] (v27) at (7,3) {$A$};
\node[nodeB] (v28) at (7,4) {$C$};

\node[nodeA] (v29) at (8,0) {${}$};
\node[nodeA] (v30) at (8,5) {${}$};

\node[nodeB] (v31) at (9,1) {$D$};
\node[nodeB] (v32) at (9,2) {$B$};
\node[nodeB] (v33) at (9,3) {$A$};
\node[nodeB] (v34) at (9,4) {$C$};


\node[nodeA] (v35) at (10,0) {${}$};
\node[nodeA] (v36) at (10,5) {${}$};

\node[nodeB] (v37) at (11,1) {$D$};
\node[nodeB] (v38) at (11,2) {$B$};
\node[nodeB] (v39) at (11,3) {$A$};
\node[nodeB] (v40) at (11,4) {$C$};

\node[nodeA] (v41) at (12,0) {${}$};
\node[nodeA] (v42) at (12,5) {${}$};

\node[nodeB] (v43) at (13,1.25) {$B$};
\node[nodeB] (v44) at (13,2.5) {$A$};
\node[nodeB] (v45) at (13,3.75) {$C$};

\foreach \a/\b in {v5/v1, v5/v2,v5/v3,v5/v4,v1/v2,v1/v3,
                   v1/v4, v6/v4,v6/v2,
                   v0/v2,v0/v4,v2/v3,v3/v4,v2/v9,v4/v9,
v12/v10, v12/v7, v12/v13, v12/v14, v12/v15, v12/v16, 
v11/v10, v11/v7, v11/v13, v11/v14, v11/v15, v11/v16,
v10/v9,  v10/v16, v10/v15, v10/v14, 
v9/v8,  v9/v15,  
v8/v7, v8/v14, 
v7/v13,v7/v14,v7/v15,
v16/v15, v15/v14, v14/v13,
v6/v7, v6/v8, v6/v9, v6/v10, 
v0/v7, v0/v8, v0/v9, v0/v10,
v4/v10, v4/v9, v4/v8, 
v2/v9, v2/v8, v2/v7, v3/v8,v6/v12,v0/v11,v5/v0,v1/v6}
{\draw (\a) -- (\b); }

\foreach \a/\b in {v17/v18, v17/v19,v17/v20,v17/v22,v18/v19,v18/v20,
                   v18/v22, v24/v22,v24/v19,
                   v23/v19,v23/v22,v19/v20,v21/v22,v19/v27,v22/v27,
v30/v28, v30/v25, v30/v31, v30/v32, v30/v33, v30/v34, 
v29/v28, v29/v25, v29/v31, v29/v32, v29/v33, v29/v34,
v28/v27,  v28/v34, v28/v33, v28/v32, 
v27/v26,  v27/v33,  
v26/v25, v26/v32, 
v25/v31,v25/v32,v25/v33,
v34/v33, v33/v32, v32/v31,
v24/v25, v24/v26, v24/v27, v24/v28, 
v23/v25, v23/v26, v23/v27, v23/v28,v20/v21,v18/v21,v17/v21,
v22/v28, v22/v27, v22/v26, 
v19/v27, v19/v26, v19/v25, v20/v26,v24/v30,v23/v29,v17/v23,v18/v24}
{\draw (\a) -- (\b); }

\foreach \a/\b in {v35/v36, v35/v37,v35/v38,v35/v40,v36/v37,v36/v38,
                   v36/v40, v42/v40,v42/v37,
                   v41/v37,v41/v40,v37/v38,v37/v44,v40/v44,
v45/v44,  v44/v43, v42/v43, v42/v44, v42/v45, 
 v41/v43, v41/v44, v41/v45,
v40/v45, v40/v44, v40/v43, 
v37/v44, v37/v43, v38/v43,v35/v41,v36/v42,v36/v39,v39/v44,v39/v40,v38/v39,
v0/v6,v11/v12,v23/v24,v29/v30,v41/v42,v21/v27}
{\draw (\a) -- (\b); }

\draw (1.14,2) .. controls (2.1,2.5) and (2.1,3.5) .. (1.14,4);
\draw (1.14,1) .. controls (2.1,1.5) and (2.1,2.5) .. (1.14,3);

\draw (3.1,2.1) .. controls (3.5,2.5) and (3.5,3.5) .. (3.1,3.9);
\draw (3.1,1.1) .. controls (3.5,1.5) and (3.5,2.5) .. (3.1,2.9);

\draw (5.14,2) .. controls (6.1,2.5) and (6.1,3.5) .. (5.14,4);
\draw (5.14,1) .. controls (6.1,1.5) and (6.1,2.5) .. (5.14,3);
\draw (7.14,2) .. controls (8.1,2.5) and (8.1,3.5) .. (7.14,4);
\draw (7.14,1) .. controls (8.1,1.5) and (8.1,2.5) .. (7.14,3);

\draw (9.1,2.1) .. controls (9.5,2.5) and (9.5,3.5) .. (9.1,3.9);
\draw (9.1,1.1) .. controls (9.5,1.5) and (9.5,2.5) .. (9.1,2.9);

\draw (11.14,2) .. controls (12.1,2.5) and (11.7,3.5) .. (11.14,3.9);
\draw (11.14,1.10) .. controls (11.7,1.5) and (12.1,2.5) .. (11.14,2.95);

\draw (13.14,1.3) .. controls (13.5,1.875) and (13.5,3.125) .. (13.14,3.7);

\draw[dashed] (v11)--(v17);
\draw[dashed] (v12)--(v18);
\draw[dashed] (v29)--(v35);
\draw[dashed] (v30)--(v36);

\draw[dashed] (v13)--(v19);
\draw[dashed] (v14)--(v20);
\draw[dashed] (v15)--(v21);
\draw[dashed] (v16)--(v22);

\draw[dashed] (v31)--(v37);
\draw[dashed] (v32)--(v38);
\draw[dashed] (v33)--(v39);
\draw[dashed] (v34)--(v40);

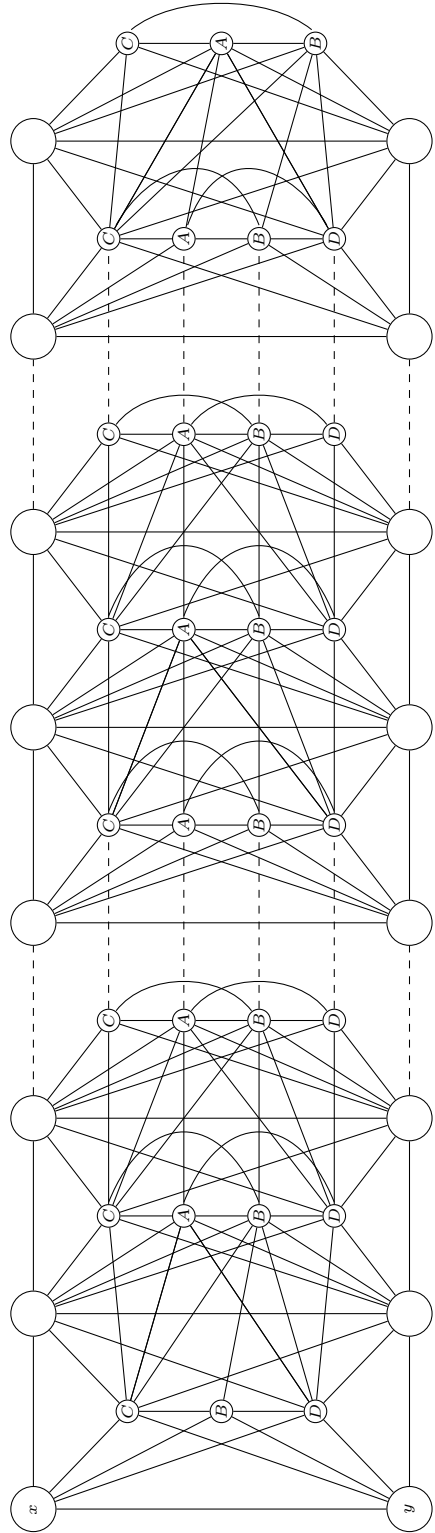
\node[draw=none,below=1.2cm of v17, font=\normalsize] {The equitable partition $U=U(x,y)$ corresponding to a pair of adjacent vertices $x$ and $y$.};

\end{tikzpicture}
\end{landscape}

\newpage

\section{A module $U$ for the subconstituent algebra $T$}

\noindent We continue to discuss the bilinear forms graph $\Gamma=(X, \mathcal R)$ and the adjacent vertices $x,y \in X$.
 In this section, we consider the subconstituent algebra $T=T(x)$ of $\Gamma$ with respect to $x$. 
Recall that $T$ is generated by the adjacency matrix $A$ of $\Gamma$ and the dual primitive idempotents $\lbrace E^*_i \rbrace_{i=0}^D$ of $\Gamma$ with respect to $x$.
We use the $(x,y)$-partition of $X$ to construct a $T$-module $U=U(x,y)$. We decompose $U$ into a direct sum of irreducible $T$-modules.
\medskip


\noindent For a subset $S \subseteq X$ define $\widehat S = \sum_{z \in S} \hat z$. We call $\widehat S$ the {\it characteristic vector} of $S$.
 
 \begin{definition} \label{def:TmodU} \rm For the subsets in the $(x,y)$-partition of $X$ their characteristic vectors form a basis for a $(6D-2)$-dimensional 
 subspace $U=U(x,y)$ of the standard module $V$. We call this
 basis the {\it  $(0/1)$-basis} for $U$. 
 \end{definition}
 
 \begin{lemma} \label{lem:Usym} With reference to Definition \ref{def:TmodU}, we have  $U(x,y)=U(y,x)$.
 \end{lemma}
 \begin{proof} By the comment above Definition \ref{def:xyPartition}.
 \end{proof}
 
 \begin{lemma}
 \label{lem:EsAction} For $0 \leq j \leq D$ the dual primitive idempotent $E^*_j$ acts on the $(0/1)$-basis for $U$ in the following way:
 \begin{align*}
& E^*_j \widehat O_{i,i-1} = \delta_{i,j} \widehat O_{i,i-1}\quad (1 \leq i \leq D);
 \qquad \qquad
 E^*_j \widehat O_{i,i}^A = \delta_{i,j} \widehat O_{i,i}^A\quad (2 \leq i \leq D);\\
 &
 E^*_j \widehat O_{i,i}^B= \delta_{i,j} \widehat O_{i,i}^B\quad (1 \leq i \leq D);
 \qquad \qquad \qquad
 E^*_j \widehat O_{i,i}^C = \delta_{i,j} \widehat O_{i,i}^C\quad (1 \leq i \leq D); \\
 &
 E^*_j \widehat O_{i,i}^D = \delta_{i,j} \widehat O_{i,i}^D\quad (1 \leq i \leq D-1);
 \qquad \qquad 
 E^*_j \widehat O_{i-1,i} = \delta_{i-1,j} \widehat O_{i-1,i}\quad (1 \leq i \leq D).
 \end{align*}
 \end{lemma}
 \begin{proof} By Proposition \ref{prop:parts}(i),(vi) and Lemma \ref{lem:disjUnion}.
 \end{proof}

 \begin{lemma} \label{lem:UisTmod} The subspace $U$ from Definition \ref{def:TmodU} is a $T$-module.
 \end{lemma}
  \begin{proof} By Proposition \ref{prop:equit} we have $A U \subseteq U$. By Lemma \ref{lem:EsAction} we have $E^*_j U \subseteq U$ for $0 \leq j \leq D$.
   By these comments $T U \subseteq U$.
 \end{proof}
 
 \noindent Our next general goal is to decompose $U$ into a direct sum of irreducible $T$-modules.
 
 \begin{lemma} \label{lem:EsU}
 We have
 \begin{align*}
 U = \sum_{i=0}^D E^*_iU \qquad \qquad \hbox{\rm (orthogonal direct sum).}
 \end{align*}
 Moreover the following {\rm (i)--(iv)} hold:
 \begin{enumerate}
 \item[\rm (i)] $E^*_0U$ has basis $\widehat O_{0,1}$;
 \item[\rm (ii)] $E^*_1U$ has basis $\widehat O_{1,0}, \widehat O^B_{1,1}, \widehat O^C_{1,1}, \widehat O^D_{1,1}, \widehat O_{1,2}$;
  \item[\rm (iii)] for $2 \leq i \leq D-1$ the subspace $E^*_iU$ has basis $\widehat O_{i,i-1}, \widehat O^A_{i,i}, \widehat O^B_{i,i}, \widehat O^C_{i,i}, \widehat O^D_{i,i}, \widehat O_{i,i+1}$;
   \item[\rm (iv)] $E^*_D U$ has basis $\widehat O_{D,D-1}, \widehat O^A_{D,D}, \widehat O^B_{D,D}, \widehat O^C_{D,D}$.
 \end{enumerate}
 \end{lemma}
\begin{proof} By Lemma  \ref{lem:EsAction}.
\end{proof}

 \begin{definition} \rm We refer to Lemma \ref{lem:EsU}. For $0 \leq i \leq D$, the given basis for $E^*_iU$  is called the {\it  $(0/1)$-basis}.
 \end{definition}

 \noindent Recall the matrices $L,F,R$ from  \eqref{eq:LFR}.
By construction,
\begin{align*}
&L E^*_i U \subseteq E^*_{i-1} U \qquad \quad (1 \leq i \leq D), \qquad L E^*_0U =0, \\
& F E^*_i U \subseteq E^*_i U \qquad \quad (0 \leq i \leq D), \\
&R E^*_i U \subseteq E^*_{i+1} U \qquad \quad (0 \leq i \leq D-1), \qquad RE^*_D U=0.
\end{align*}

\begin{lemma} \label{lem:Frep} For $0 \leq i \leq D$ we give the matrix ${\sf F}_i$ that represents the action of $F$ on $E^*_iU$ with respect to the $(0/1)$-basis.
We have ${\sf F}_0=0$ and
\begin{align*}
{\sf F}_1=\begin{pmatrix} 0&q-2&q^{N-D}-q&q^D-q&0 \\
                         1&q-3&q^{N-D}-q&q^D-q&0 \\
                          1&q-2&q^{N-D}-q-1&0&q^D-q \\
                         1 &q-2&0&q^D-q-1&q^{N-D}-q \\
                          0&0&q-1&q-1&q^{N-D}+q^D-3q
                          \end{pmatrix}.
 \end{align*}
 \noindent For $2 \leq i \leq D-1$ we have ${\sf F}_i=$
\begin{tiny}
 \begin{align*}
  \begin{pmatrix}  f^{(1)}_i &(q^{i-1}-1)(q^{i-1}-q^{i-2})&q^{i-2}(q-2)(q^i-q^{i-1}+1)&q^{i-1} (q^{N-D}-q^i)&q^{i-1} (q^D-q^i) &0 \\
                          q^{i-1}(q^{i-1}-q^{i-2}) &f^{(2)}_i&q^{2i-3}(q-1)(q-2)&q^{i-1}(q^{N-D}-q^i)&q^{i-1}(q^D-q^i)&0\\
                           q^{i-2}(q^i-q^{i-1}+1)&q^{i-2}(q-1)(q^{i-1}-1)&f^{(3)}_i&q^{i-1}(q^{N-D}-q^i)&q^{i-1}(q^D-q^i)&0\\
                          q^{2i-2} &q^{i-1}(q^{i-1}-1)&q^{2i-2}(q-2)&f^{(4)}_i&0&q^{i-1}(q^D-q^i)\\
                          q^{2i-2}&q^{i-1}(q^{i-1}-1)&q^{2i-2}(q-2)&0&f^{(5)}_i&q^{i-1}(q^{N-D}-q^i)\\
                           0&0&0&q^{i-1}(q^i-1)&q^{i-1}(q^i-1)&f^{(6)}_i
                           \end{pmatrix},
\end{align*} 
\end{tiny}
\noindent where
\begin{align*}
f^{(1)}_i &= \frac{q^{i-1}-1}{q-1} \Bigl( q^{N-D}+q^D-3q^{i-1}+q^{i-2}-1     \Bigr), \\
f^{(2)}_i &=  \frac{q^{i-1}-1}{q-1} \Bigl( q^{N-D}+q^D-3q^{i-1}+q^{i-2}-2     \Bigr) + \frac{q^{i-2}-1}{q-1}, \\
f^{(3)}_i &=  \frac{q^{i-1}-1}{q-1} \Bigl( q^{N-D}+q^D+q^{i+1}-5q^i+4q^{i-1}-2q^{i-2}+q^2-4q+2     \Bigr) + q-3, \\
f^{(4)}_i &= \frac{ q^i-1}{q-1} \Bigl( q^{N-D}-q^i-1\Bigr) + \frac{q^{i-1}-1}{q-1} \Bigl( q^D-q^i\Bigr), \\
f^{(5)}_i &= \frac{ q^i-1}{q-1} \Bigl( q^{D}-q^i-1\Bigr) + \frac{q^{i-1}-1}{q-1} \Bigl( q^{N-D}-q^i\Bigr), \\
f^{(6)}_i &= \frac{q^{i}-1}{q-1} \Bigl( q^{N-D}+q^D-3q^{i}+q^{i-1}-1     \Bigr).
\end{align*}
We have ${\sf F}_D=$
\begin{footnotesize}
\begin{align*}
\begin{pmatrix}f^{(1)}_D &(q^{D-1}-1)(q^{D-1}-q^{D-2})&q^{D-2}(q-2)(q^D-q^{D-1}+1)&q^{D-1} (q^{N-D}-q^D)\\
                          q^{D-1}(q^{D-1}-q^{D-2}) &f^{(2)}_D&q^{2D-3}(q-1)(q-2)&q^{D-1}(q^{N-D}-q^D)\\
                           q^{D-2}(q^D-q^{D-1}+1)&q^{D-2}(q-1)(q^{D-1}-1)&f^{(3)}_D&q^{D-1}(q^{N-D}-q^D)\\
                          q^{2D-2} &q^{D-1}(q^{D-1}-1)&q^{2D-2}(q-2)&f^{(4)}_D\\
                          \end{pmatrix}.
 \end{align*}
 \end{footnotesize}
\end{lemma}
\begin{proof}
By Lemmas \ref{lem:structure1}--\ref{lem:structure6}.
\end{proof}
 
 \begin{lemma} \label{lem:Rrep} For $0 \leq i \leq D-1$ we give the matrix ${\sf R}_i$ that represents the action of $R:E^*_iU \to E^*_{i+1}U$ with respect to the $(0/1)$-bases.
 We have
  \begin{align*}
{\sf R}_0 =  \begin{pmatrix} 1&1&1&1&1 
                          \end{pmatrix}^t
 \end{align*}
and
 \begin{align*}
{\sf R}_1 =  \begin{pmatrix} 1&0&q-1&q-1&q^2-q+1\\
                          0&0&q&q&q(q-1)\\
                         0& 1& q-1&q-1& q^2-q+1\\
                          0&0&q&0&q^2\\
                          0&0&0&q&q^2\\
                          0&0&0&0&q(q+1)
                          \end{pmatrix}.
 \end{align*}
\noindent  For $2 \leq i \leq D-2$ we have
 \begin{align*}
{\sf R}_i =  \begin{pmatrix} q^{i-1}\frac{q^i-1}{q-1}&0&0&q^{i-1}(q^i-1)&q^{i-1}(q^i-1)&q^{2i}-q^{2i-1}+q^{i-1}\\
                          0&q^{i}\frac{q^{i-1}-1}{q-1}&0&q^{2i-1}&q^{2i-1}&q^i(q^i-q^{i-1})\\
                         0&0& q^{i-1}\frac{q^i-1}{q-1}&q^{i-1}(q^i-1)&q^{i-1}(q^i-1)&q^{i-1}(q^{i+1}-q^i+1)\\
                          0&0&0&q^{i}\frac{q^i-1}{q-1}&0&q^{2i}\\
                          0&0&0&0&q^{i}\frac{q^i-1}{q-1}&q^{2i}\\
                          0&0&0&0&0&q^{i}\frac{q^{i+1}-1}{q-1}
                          \end{pmatrix}.
 \end{align*}
 \noindent We have ${\sf R}_{D-1} =$
 \begin{footnotesize}
  \begin{align*}
\begin{pmatrix} 
 q^{D-2}\frac{q^{D-1}-1}{q-1}&0&0&q^{D-2}(q^{D-1}-1)&q^{D-2}(q^{D-1}-1)&q^{2D-2}-q^{2D-3}+q^{D-2}\\
                          0&q^{D-1}\frac{q^{D-2}-1}{q-1}&0&q^{2D-3}&q^{2D-3}&q^{D-1}(q^{D-1}-q^{D-2})\\
                          0&0& q^{D-2}\frac{q^{D-1}-1}{q-1}&q^{D-2}(q^{D-1}-1)&q^{D-2}(q^{D-1}-1)&q^{D-2}(q^{D}-q^{D-1}+1)\\
                          0&0&0&q^{D-1}\frac{q^{D-1}-1}{q-1}
&0&q^{2D-2}\\
                          \end{pmatrix}.
 \end{align*}
 \end{footnotesize}
 \end{lemma}
 \begin{proof}
By Lemmas \ref{lem:structure1}--\ref{lem:structure6}.
\end{proof}

 \begin{lemma} For $1 \leq i \leq D$ we give the  matrix ${\sf L}_i$ that represents the action of $L:E^*_iU \to E^*_{i-1}U$ with respect to the $(0/1)$-bases.
 We have
 \begin{landscape}
  \begin{align*} {\sf L}_1= \begin{pmatrix} 
 1&q-2&q^{N-D}-q&q^D-q&\frac{(q^{N-D}-q)(q^D-q)}{q-1}\end{pmatrix}
\end{align*} 
 and 
 \begin{align*} 
 {\sf L}_2=
 \begin{pmatrix} 
          \frac{(q^{N-D}-q)(q^D-q)}{q-1} &0&0&0&0&0\\
          0 &0& \frac{(q^{N-D}-q)(q^D-q)}{q-1}&0&0&0\\
          q^D-q &q^D-q&(q-2)(q^D-q)& \frac{(q^{N-D}-q^2)(q^D-q)}{q-1}&0&0\\
          q^{N-D}-q&q^{N-D}-q&(q-2)(q^{N-D}-q)&0& \frac{(q^{N-D}-q)(q^D-q^2)}{q-1}&0\\
         q^2-q+1   &(q-1)^2&(q-2)(q^2-q+1)   &q(q^{N-D}-q^2)&q(q^D-q^2)& \frac{(q^{N-D}-q^2)(q^D-q^2)}{q-1}\end{pmatrix}.
\end{align*}
\noindent For $3 \leq i \leq D-1$ we have 
\begin{footnotesize}
\begin{align*}
{\sf L}_i = 
 \begin{pmatrix} 
          \frac{(q^{N-D}-q^{i-1})(q^D-q^{i-1})}{q-1} &0&0&0&0&0\\
           0& \frac{(q^{N-D}-q^{i-1})(q^D-q^{i-1})}{q-1}&0&0&0&0\\
          0 &0& \frac{(q^{N-D}-q^{i-1})(q^D-q^{i-1})}{q-1}&0&0&0\\
          q^{i-2}(q^D-q^{i-1}) &q^{i-2}(q^D-q^{i-1})&q^{i-2}(q-2)(q^D-q^{i-1})& \frac{(q^{N-D}-q^{i})(q^D-q^{i-1})}{q-1}&0&0\\
          q^{i-2}(q^{N-D}-q^{i-1})&q^{i-2}(q^{N-D}-q^{i-1})&q^{i-2}(q-2)(q^{N-D}-q^{i-1})&0& \frac{(q^{N-D}-q^{i-1})(q^D-q^{i})}{q-1}&0\\
         q^{2i-2}-q^{2i-3}+q^{i-2}   &(q^{i-1}-1)(q^{i-1}-q^{i-2})&q^{i-2}(q-2)(q^i-q^{i-1}+1)   &q^{i-1}(q^{N-D}-q^{i})&q^{i-1}(q^D-q^{i})& \frac{(q^{N-D}-q^{i})(q^D-q^{i})}{q-1}
\end{pmatrix}.
\end{align*}
\end{footnotesize}
\noindent We have 
\begin{align*}
{\sf L}_D=
 \begin{pmatrix} 
           (q^{N-D}-q^{D-1})q^{D-1} &0&0&0\\
           0& (q^{N-D}-q^{D-1})q^{D-1} &0&0\\
          0 &0& (q^{N-D}-q^{D-1})q^{D-1} &0\\
          q^{D-2}(q^D-q^{D-1}) &q^{D-2}(q^D-q^{D-1})&q^{D-2}(q-2)(q^D-q^{D-1})& (q^{N-D}-q^{D})q^{D-1} \\
          q^{D-2}(q^{N-D}-q^{D-1})&q^{D-2}(q^{N-D}-q^{D-1})&q^{D-2}(q-2)(q^{N-D}-q^{D-1})&0\\
         q^{2D-2}-q^{2D-3}+q^{D-2}   &(q^{D-1}-1)(q^{D-1}-q^{D-2})&q^{D-2}(q-2)(q^D-q^{D-1}+1)   &q^{D-1}(q^{N-D}-q^{D})
\end{pmatrix}.
\end{align*}
\end{landscape}
\end{lemma}
\begin{proof}
By Lemmas \ref{lem:structure1}--\ref{lem:structure6}.
\end{proof}

\noindent Next, for $0 \leq i \leq D$ we diagonalize the matrix  ${\sf F}_i$ from Lemma \ref{lem:Frep}. This will be done using a matrix ${\sf H}_i$ that we now define.

\begin{definition} \label{def:Hi} \rm For $0 \leq i \leq D$ we define a matrix ${\sf H}_i$ as follows. We  have ${\sf H}_0=(1)$ and
\begin{align*}
{\sf H}_1= \begin{pmatrix} 1&\frac{q^D-q}{q-1}&\frac{q^{N-D}-q}{q-1}&\frac{(q^D-1)(q^{N-D}-1)(q-2)}{(q-1)^2}& \frac{(q^D-q)(q^{N-D}-q)}{(q-1)^2}\\
                               1&\frac{q^D-q}{q-1}&\frac{q^{N-D}-q}{q-1}&-\frac{(q^D-1)(q^{N-D}-1)}{(q-1)^2}& \frac{(q^D-q)(q^{N-D}-q)}{(q-1)^2}\\
                               1&\frac{q^D-q}{q-1}&-1&0&- \frac{q^D-q}{q-1}\\
                               1&-1&\frac{q^{N-D}-q}{q-1}&0&-\frac{q^{N-D}-q}{q-1}\\
                               1&-1&-1&0&1
                               \end{pmatrix}.
\end{align*}
For $2 \leq i \leq D-1$ we have
\begin{footnotesize}
 \begin{align*}
 {\sf H}_i = \begin{pmatrix} 1&q^{1-i}\frac{q^D-q^i}{q-1}&q^{1-i}\frac{q^{N-D}-q^i}{q-1}&\frac{(q-2)(q^D-1)(q^{N-D}-1)}{q^{i-1}(q-1)^2}&\frac{q^D-q^i}{q-1} \frac{q^{N-D}-q^i}{q-1}&-\frac{q^D-1}{q-1} \frac{q^{N-D}-1}{q-1}\frac{q^{i-1}-1}{q-1}\\
                                 1&q^{1-i}\frac{q^D-q^i}{q-1}&q^{1-i}\frac{q^{N-D}-q^i}{q-1}&0&\frac{q^D-q^i}{q-1}\frac{q^{N-D}-q^i}{q-1}&q^{i-1} \frac{q^D-1}{q-1} \frac{q^{N-D}-1}{q-1}\\
                                 1&q^{1-i}\frac{q^D-q^i}{q-1}&q^{1-i}\frac{q^{N-D}-q^i}{q-1}&-\frac{(q^D-1)(q^{N-D}-1)}{q^{i-1}(q-1)^2}&\frac{q^D-q^i}{q-1}\frac{q^{N-D}-q^i}{q-1}&-\frac{q^D-1}{q-1} \frac{q^{N-D}-1}{q-1}\frac{q^{i-1}-1}{q-1}\\
                                 1&q^{1-i}\frac{q^D-q^i}{q-1}&-q^{1-i} \frac{q^i-1}{q-1}&0&-\frac{q^D-q^i}{q-1}\frac{q^i-1}{q-1}&0\\
                                 1&-q^{1-i} \frac{q^i-1}{q-1}&q^{1-i}\frac{q^{N-D}-q^i}{q-1}&0&-\frac{q^i-1}{q-1}\frac{q^{N-D}-q^i}{q-1}&0\\
                                 1&-q^{1-i} \frac{q^i-1}{q-1}&-q^{1-i} \frac{q^i-1}{q-1}&0&\Bigl(\frac{q^i-1}{q-1}\Bigr)^2&0
                                 \end{pmatrix}.
                                 \end{align*}
                                 \end{footnotesize}
\noindent We have
\begin{align*}
{\sf H}_D= \begin{pmatrix} 1&q^{1-D}\frac{q^{N-D}-q^D}{q-1}&\frac{(q-2)(q^D-1)(q^{N-D}-1)}{q^{D-1}(q-1)^2}&-\frac{q^D-1}{q-1} \frac{q^{N-D}-1}{q-1}\frac{q^{D-1}-1}{q-1}\\
                                 1&q^{1-D}\frac{q^{N-D}-q^D}{q-1}&0&q^{D-1} \frac{q^D-1}{q-1} \frac{q^{N-D}-1}{q-1}\\
                                 1&q^{1-D}\frac{q^{N-D}-q^D}{q-1}&-\frac{(q^D-1)(q^{N-D}-1)}{q^{D-1}(q-1)^2}&-\frac{q^D-1}{q-1} \frac{q^{N-D}-1}{q-1}\frac{q^{D-1}-1}{q-1}\\
                                 1&-q^{1-D} \frac{q^D-1}{q-1}&0&0
                               \end{pmatrix}.
\end{align*}            
                                                     \end{definition}
                                 
 \begin{lemma} \label{lem:Hinv} For $0 \leq i \leq D$ the matrix ${\sf H}_i$ is invertible, and ${\sf H}^{-1}_i $ is described as follows. We have
 ${\sf H}^{-1}_0 = (1)$ and
 \begin{align*}
 {\sf H}^{-1}_1 = 
 \frac{1}{k} \begin{pmatrix}
   1&q-2&q^{N-D}-q&q^D-q& \frac{(q^{N-D}-q)(q^D-q)}{q-1}\\
   1&q-2&q^{N-D}-q&1-q&q-q^{N-D}\\
   1&q-2&1-q&q^D-q&q-q^D\\
   1&-1&0&0&0\\
   1&q-2&1-q&1-q&q-1
 \end{pmatrix}.
 \end{align*}
 For $2 \leq i \leq D-1$ we have ${\sf H}_i^{-1} = $
 \begin{footnotesize}
 \begin{align*}
 \frac{1}{k}\begin{pmatrix}
 q^{i-1} \frac{q^i-1}{q-1} &\frac{(q^i-1)(q^{i-1}-1)}{q-1}&(q-2)q^{i-1}\frac{q^i-1}{q-1}&\frac{(q^i-1)(q^{N-D}-q^i)}{q-1}&\frac{(q^i-1)(q^D-q^i)}{q-1}&\frac{(q^D-q^i)(q^{N-D}-q^i)}{q-1}\\
 q^{2i-2}&q^{i-1}(q^{i-1}-1)&(q-2)q^{2i-2} &q^{i-1}(q^{N-D}-q^i)&-q^{i-1}(q^i-1)&-q^{i-1}(q^{N-D}-q^i)\\
  q^{2i-2}&q^{i-1}(q^{i-1}-1)&(q-2)q^{2i-2} &-q^{i-1}(q^i-1)&q^{i-1}(q^D-q^i)&-q^{i-1}(q^D-q^i)\\
 q^{i-1} &0&-q^{i-1}&0&0&0\\
 q^{i-1}\frac{q-1}{q^i-1}&\frac{(q-1)(q^{i-1}-1)}{q^i-1}&\frac{(q-1)(q-2)q^{i-1}}{q^i-1}&1-q&1-q&q-1\\
-\frac{q-1}{q^i-1} &\frac{(q-1)^2}{q^i-1}&- \frac{(q-1)(q-2)}{q^i-1}&0&0&0
 \end{pmatrix}.
 \end{align*}
 \end{footnotesize}
  We have
   \begin{align*}
 {\sf H}^{-1}_D = 
 \frac{1}{k} \begin{pmatrix}
 q^{D-1} \frac{q^D-1}{q-1} &\frac{(q^D-1)(q^{D-1}-1)}{q-1}&(q-2)q^{D-1}\frac{q^D-1}{q-1}&\frac{(q^D-1)(q^{N-D}-q^D)}{q-1}\\
  q^{2D-2}&q^{D-1}(q^{D-1}-1)&(q-2)q^{2D-2} &-q^{D-1}(q^D-1)\\
 q^{D-1} &0&-q^{D-1}&0\\
-\frac{q-1}{q^D-1} &\frac{(q-1)^2}{q^D-1}&- \frac{(q-1)(q-2)}{q^D-1}&0
 \end{pmatrix}.
 \end{align*}
 \end{lemma}    
 \begin{proof}                            
 By Definition \ref{def:Hi} and linear algebra.
 \end{proof}

 \begin{lemma} \label{lem:Hdiag} For $0 \leq i \leq D$ the product ${\sf H}^{-1}_i {\sf F}_i {\sf H}_i = {\sf D}_i$, where ${\sf D}_i$ is the following diagonal matrix.
 We have ${\sf D}_0=0$ and
 \begin{align*}
 {\sf D}_1 = {\rm diag}\bigl( q^{N-D}+q^D-q-2, q^{N-D}-q-1, q^D-q-1, -1,-q \bigr).
 \end{align*}
 For $2 \leq i \leq D-1$ we have
 \begin{align*}
 {\sf D}_i = {\rm diag}\bigl(\lambda^{(1)}_i, \lambda^{(2)}_i, \lambda^{(3)}_i, \lambda^{(4)}_i, \lambda^{(5)}_i, \lambda^{(6)}_i \bigr)
 \end{align*}
 where
 \begin{align*}
 \lambda^{(1)}_i &= a_i =    \frac{q^i -1}{q-1} \Bigl           (q^{N-D}  +q^D   -q^i - q^{i-1}-1\Bigr),\\
 \lambda^{(2)}_i &= a_i - q^{i-1} (q^D-1) \\
 &= \frac{q^i-1}{q-1} \Bigl(q^{N-D}-q^i-1\Bigr) + \frac{q^{i-1}-1}{q-1} \Bigl( q^D-q^i \Bigr ), \\
 \lambda^{(3)}_i &= a_i - q^{i-1} (q^{N-D}-1)\\
 &=  \frac{q^{i-1}-1}{q-1} \Bigl(q^{N-D}-q^i\Bigr) + \frac{q^i-1}{q-1} \Bigl( q^D-q^i -1 \Bigr ), 
  \\
 \lambda^{(4)}_i &= a_i - q^{i-1} \bigl(q^{N-D}+q^D-q^i-q^{i-1}  \bigr) \\
 &=  \frac{q^{i-1}-1}{q-1} \Bigl(q^{N-D}+q^D-q^i-q^{i-1}\Bigr) -\frac{q^i
-1}{q-1} 
 , \\
 \lambda^{(5)}_i &= a_i - q^{i-1} \bigl( q^{N-D}+q^D-2\bigr) \\
 &=  \frac{q^{i-1}-1}{q-1} \Bigl(q^{N-D}+q^D-q^i-1\Bigr) -q^i\frac{q^i
-1}{q-1} , \\
 \lambda^{(6)}_i &= a_i - q^{i-1} \bigl( q^{N-D}+q^D-q^i-q^{i-1}-1+q^{-1}  \bigr) \\
 &=  \frac{q^{i-1}-1}{q-1} \Bigl(q^{N-D}+q^D-q^i-1\Bigr) -q^{i-2}\frac{q^i
-1}{q-1}.
 \end{align*}
\noindent We have
  \begin{align*}
 {\sf D}_D = {\rm diag}\bigl(\lambda^{(1)}_D, \lambda^{(3)}_D, \lambda^{(4)}_D, \lambda^{(6)}_D \bigr).
 \end{align*}
 \end{lemma}
  \begin{proof}                            
 By  linear algebra.
 \end{proof}
 
 \begin{lemma} \label{lem:HiRH} For $0 \leq i \leq D-1$ the product ${\sf H}^{-1}_{i+1} {\sf R}_i {\sf H}_i$ is given as follows.
We have
   \begin{align*}
 {\sf H}^{-1}_{1} {\sf R}_0 {\sf H}_0= 
 \begin{pmatrix}
	1&0&0&0&0 \\
 \end{pmatrix}^t
 \end{align*}
and
 \begin{align*}
 {\sf H}^{-1}_{2} {\sf R}_1 {\sf H}_1= 
 \begin{pmatrix}
          q(q+1) &0&0&0&0 \\
           0&q^2&0&0&0\\
           0&0&q^2&0&0\\
          0&0&0&q&0\\
           0&0&0&0&\frac{q}{q+1}\\
           0&0&0&0&\frac{-1}{q+1}
 \end{pmatrix}.
 \end{align*}
 \noindent
 For $2 \leq i \leq D-2$ we have
 \begin{align*}
 {\sf H}^{-1}_{i+1} {\sf R}_i {\sf H}_i= 
 \begin{pmatrix}
          q^i\frac{q^{i+1}-1}{q-1} &0&0&0&0&0 \\
           0&q^{i+1}\frac{q^i-1}{q-1}&0&0&0&0\\
           0&0&q^{i+1}\frac{q^i-1}{q-1}&0&0&0\\
          0 &0&0&q^i\frac{q^i-1}{q-1}&0&0\\
           0&0&0&0&q^i\frac{(q^i-1)^2}{(q^{i+1}-1)(q-1)}&0\\
           0&0&0&0&-q^{i-1}\frac{q-1}{q^{i+1}-1}&q^{i-1} \frac{q^{i-1}-1}{q-1}
 \end{pmatrix}.
 \end{align*}
  We have
  \begin{align*}
 {\sf H}^{-1}_{D} {\sf R}_{D-1} {\sf H}_{D-1}= 
 \begin{pmatrix}
         q^{D-1}\frac{q^D-1}{q-1} &0&0&0&0&0 \\
           0&0&q^D\frac{q^{D-1}-1}{q-1}&0&0&0\\
          0 &0&0&q^{D-1}\frac{q^{D-1}-1}{q-1}&0&0\\
           0&0&0&0&-q^{D-2}\frac{q-1}{q^D-1}&q^{D-2} \frac{q^{D-2}-1}{q-1}
 \end{pmatrix}.
 \end{align*}
 \end{lemma}
   \begin{proof}                            
 By  linear algebra.
 \end{proof}


  \begin{lemma} \label{lem:HiLH} For $1 \leq i \leq D$ the product ${\sf H}^{-1}_{i-1} {\sf L}_i {\sf H}_i$ is given as follows.
 We have
 \begin{landscape}
   \begin{align*}
{\sf H}^{-1}_{0} {\sf L}_1 {\sf H}_1 = \begin{pmatrix}
  \frac{(q^D-1)(q^{N-D}-1)}{q-1}&0&0&0&0\\
\end{pmatrix}
\end{align*}
\noindent and 
  \begin{align*}
  {\sf H}^{-1}_{1} {\sf L}_2 {\sf H}_2 =
  \begin{pmatrix}
 \frac{(q^D-q)(q^{N-D}-q)}{q-1}&0&0&0&0&0\\
0&\frac{(q^D-q^2)(q^{N-D}-q)}{q(q-1)}&0&0&0&0\\
0&0&\frac{(q^D-q)(q^{N-D}-q^2)}{q(q-1)}&0&0&0\\
0&0&0&\frac{(q^D-q)(q^{N-D}-q)}{q(q-1)}&0&0\\
0&0&0&0&\frac{(q^D-q^2)(q^{N-D}-q^2)}{q-1}&-\frac{(q^D-1)(q^{N-D}-1)}{q-1}\\
\end{pmatrix}.
\end{align*}
 \noindent
 For $3 \leq i \leq D-1$ we have 
 \begin{footnotesize}
 \begin{align*}
{\sf H}^{-1}_{i-1} {\sf L}_i {\sf H}_i=
\begin{pmatrix}
\frac{(q^D-q^{i-1})(q^{N-D}-q^{i-1})}{q-1}&0&0&0&0&0\\
0&\frac{(q^D-q^{i})(q^{N-D}-q^{i-1})}{q(q-1)}&0&0&0&0\\
0&0&\frac{(q^D-q^{i-1})(q^{N-D}-q^{i})}{q(q-1)}&0&0&0\\
0&0&0&\frac{(q^D-q^{i-1})(q^{N-D}-q^{i-1})}{q(q-1)}&0&0\\
0&0&0&0&\frac{(q^D-q^{i})(q^{N-D}-q^{i})}{q-1}&-q^{i-2}\frac{(q^D-1)(q^{N-D}-1)}{q^{i-1}-1}\\
0&0&0&0&0&\frac{(q^i-1)(q^D-q^{i-1})(q^{N-D}-q^{i-1})}{(q^{i-1}-1)(q-1)}
\end{pmatrix}.
\end{align*}
\end{footnotesize}
\noindent We have 
 \begin{align*}
{\sf H}^{-1}_{D-1} {\sf L}_D {\sf H}_D =
  \begin{pmatrix}
  q^{D-1}(q^{N-D}-q^{D-1})&0&0&0\\
0&0&0&0\\
0&q^{D-2}(q^{N-D}-q^D)&0&0\\
0&0&q^{D-2}(q^{N-D}-q^{D-1})&0\\
0&0&0&-q^{D-2}\frac{(q^D-1)(q^{N-D}-1)}{q^{D-1}-1}\\
0&0&0&q^{D-1}\frac{(q^D-1)(q^{N-D}-q^{D-1})}{q^{D-1}-1}
\end{pmatrix}.
\end{align*}
\end{landscape}
 \end{lemma}
  \begin{proof}                            
 By  linear algebra.
 \end{proof}

 \noindent We are now ready to decompose the $T$-module $U$ into a direct sum of irreducible $T$-modules.
 
 \begin{theorem}\label{thm:Udecomp} The $T$-module $U$ is an orthogonal direct sum of irreducible $T$-modules
 \begin{align*}
   U = W_1 + W_2 + W_3 + W_4 +W_5.
   \end{align*}
 The $\lbrace W_i \rbrace_{i=1}^5$ are described as follows.
 \begin{enumerate}
\item[\rm (i)] $W_1$ is the primary $T$-module. It is thin, with endpoint $0$ and diameter $D$. It has a basis $\lbrace w_i \rbrace_{i=0}^D$ such that
\begin{align*}
 w_i & \in E^*_i U \qquad \qquad (0 \leq i \leq D), \\
 R w_i &= c_{i+1} w_{i+1} \qquad (0 \leq i \leq D-1), \qquad Rw_D=0, \\
 L w_i &=  b_{i-1} w_{i-1} \qquad (1 \leq i \leq D), \qquad Lw_0=0,\\
 F w_i &= a_i w_i \qquad \qquad (0 \leq i \leq D).
 \end{align*}
 \item[\rm (ii)] $W_2$ is thin, with endpoint $1$, diameter $D-2$, and local eigenvalue $q^{N-D}-q-1$. It has a basis $\lbrace w_i \rbrace_{i=1}^{D-1}$ such that
 \begin{align*}
 w_i & \in E^*_i U \qquad \qquad (1 \leq i \leq D-1), \\
 R w_i &= q^{i+1} \frac{q^i-1}{q-1} w_{i+1} \qquad (1 \leq i \leq D-2), \qquad Rw_{D-1}=0, \\
 L w_i &=  
 \frac{(q^D-q^i)(q^{N-D}-q^{i-1})}{q(q-1)} w_{i-1} \qquad (2 \leq i \leq D-1), \qquad Lw_1=0,\\
 F w_i &= \Bigl(a_i - q^{i-1} \bigl(q^D-1\bigr) \Bigr)w_i \qquad \qquad (1 \leq i \leq D-1).
 \end{align*}
 \item[\rm (iii)] $W_3$ is thin, with endpoint $1$, diameter $D-1$, and local eigenvalue $q^D-q-1$.  It has a basis $\lbrace w_i \rbrace_{i=1}^{D}$ such that
 \begin{align*}
 w_i & \in E^*_i U \qquad \qquad (1 \leq i \leq D), \\
 R w_i &= q^{i+1} \frac{q^i-1}{q-1} w_{i+1} \qquad (1 \leq i \leq D-1), \qquad Rw_{D}=0, \\
 L w_i &=  
 \frac{(q^D-q^{i-1})(q^{N-D}-q^{i})}{q(q-1)} w_{i-1} \qquad (2 \leq i \leq D), \qquad Lw_1=0,\\
 F w_i &= \Bigl(a_i - q^{i-1} \bigl(q^{N-D}-1\bigr) \Bigr)w_i \qquad \qquad (1 \leq i \leq D).
 \end{align*}
\item[\rm (iv)] $W_4$ is thin, with endpoint $1$, diameter $D-1$, and local eigenvalue $-1$. It has a basis $\lbrace w_i \rbrace_{i=1}^{D}$ such that
 \begin{align*}
 w_i & \in E^*_i U \qquad \qquad (1 \leq i \leq D), \\
 R w_i &= q^{i} \frac{q^i-1}{q-1} w_{i+1} \qquad (1 \leq i \leq D-1), \qquad Rw_{D}=0, \\
 L w_i &=  
 \frac{(q^D-q^{i-1})(q^{N-D}-q^{i-1})}{q(q-1)} w_{i-1} \qquad (2 \leq i \leq D), \qquad Lw_1=0,\\
 F w_i &= \Bigl(a_i - q^{i-1} \bigl(q^{N-D}+q^D-q^i-q^{i-1}\bigr) \Bigr)w_i \qquad \qquad (1 \leq i \leq D).
 \end{align*}
\item[\rm (v)] $W_5$ is nonthin, with endpoint $1$, diameter $D-1$, and local eigenvalue $-q$. It has a basis $\lbrace w^-_i \rbrace_{i=1}^{D-1} \cup  \lbrace w^+_i \rbrace_{i=2}^{D}$ such that
 \begin{align*}
 w^-_i & \in E^*_i U \qquad  (1 \leq i \leq D-1), \qquad \qquad  w^+_i  \in E^*_i U \qquad  (2 \leq i \leq D),\\
 R w^-_i &=  \frac{q^i(q^i-1)^2}{(q-1)(q^{i+1}-1)} w^-_{i+1} - \frac{q^{i-1}(q-1)}{q^{i+1}-1} w^+_{i+1}\qquad (1 \leq i \leq D-2), \\
  Rw^-_{D-1}&= -q^{D-2} \frac{q-1}{q^{D}-1} w^+_{D}, \\
  Rw^+_i &= q^{i-1} \frac{q^{i-1}-1}{q-1} w^+_{i+1} \qquad (2\leq i \leq D-1), \qquad \quad Rw^+_D=0,\\
 L w^-_i &=  
 \frac{(q^D-q^i)(q^{N-D}-q^i)}{q-1} w^-_{i-1} \qquad (2 \leq i \leq D-1), \qquad Lw^-_1=0,\\
 L w^+_i &= \frac{(q^i-1)(q^D-q^{i-1})(q^{N-D}-q^{i-1})}{(q-1)(q^{i-1}-1)} w^+_{i-1} - \frac{q^{i-2} (q^D-1)(q^{N-D}-1)}{q^{i-1}-1} w^-_{i-1}\\
 & \qquad \qquad \qquad \qquad \qquad \qquad \qquad \qquad \qquad (3 \leq i \leq D),\\
  L w^+_2 &= - \frac{ (q^D-1)(q^{N-D}-1)}{q-1} w^-_1, \\
 F w^-_i &= \Bigl(a_i - q^{i-1} \bigl(q^{N-D}+q^D-2\bigr) \Bigr)w^-_i \qquad \qquad (1 \leq i \leq D-1), \\
 F w^+_i &= \Bigl(a_i - q^{i-1} \bigl(q^{N-D}+q^D-q^i-q^{i-1}-1+q^{-1}\bigr) \Bigr)w^+_i \qquad \qquad (2 \leq i \leq D). 
 \end{align*}
 \end{enumerate}
 \end{theorem}
 \begin{proof} Let $0 \leq i \leq D$. In Definition \ref{def:Hi}  we introduced a matrix ${\sf H}_i$, and in Lemma \ref{lem:Hinv}  we  showed that ${\sf H}_i$ is invertible.
 By Lemma \ref{lem:Hdiag}, we may interpret ${\sf H}_i$ as a transition matrix from the $(0/1)$-basis for $E^*_iU$ to an $F_i$-eigenbasis for $E^*_iU$. We denote this $F_i$-eigenbasis as follows:

 \begin{align*} 
\begin{tabular}[t]{c|c}
{\rm case }& {\rm $F_i$-eigenbasis}
 \\
 \hline
$i=0$ & $w^{(1)}_0$ \\
$i=1$ & $w^{(1)}_1, w^{(2)}_1,  w^{(3)}_1,  w^{(4)}_1,  w^{-}_1$          \\
$2 \leq i \leq D-1$ & $w^{(1)}_i, w^{(2)}_i,  w^{(3)}_i,  w^{(4)}_i,  w^{-}_i, w^{+}_i$          \\
$i=D$ & $w^{(1)}_D, w^{(3)}_D,   w^{(4)}_D,  w^{+}_D$          \\
    \end{tabular}
\end{align*}
The following hold by Lemmas \ref{lem:Hdiag}--\ref{lem:HiLH} and the irreducibility assertion in \cite[Theorem~2.2]{hobart}: 
\begin{enumerate}
\item[\rm (i)] the vectors $\lbrace w^{(1)}_i\rbrace_{i=0}^D$ form a basis for an irreducible $T$-module that meets the description of $W_1$;
\item[\rm (ii)] the vectors $\lbrace w^{(2)}_i\rbrace_{i=1}^{D-1}$ form a basis for an irreducible $T$-module that meets the description of $W_2$;
\item[\rm (iii)] the vectors $\lbrace w^{(3)}_i\rbrace_{i=1}^D$ form a basis for an irreducible $T$-module that meets the description of $W_3$;
\item[\rm (iv)] the vectors $\lbrace w^{(4)}_i\rbrace_{i=1}^D$ form a basis for an irreducible $T$-module that meets the description of $W_4$.
\end{enumerate}
The following holds by Lemmas \ref{lem:Hdiag}--\ref{lem:HiLH} and \cite[Section~6.1]{hobart}:
\begin{enumerate}
\item[\rm (v)] the vectors $\lbrace w^{-}_i\rbrace_{i=1}^{D-1} \cup \lbrace w^{+}_i\rbrace_{i=2}^{D}$ form a basis for an irreducible $T$-module that meets the description of $W_5$.
\end{enumerate}
By construction, the sum $U=\sum_{i=1}^5 W_i$ is direct. This direct sum is orthogonal because the summands are mutually nonisomorphic, and any two
nonisomorphic irreducible $T$-modules are orthogonal.
 \end{proof}

 \section{Some comments about the subspace $U$}

\noindent We continue to discuss the bilinear forms graph $\Gamma=(X, \mathcal R)$ and the adjacent vertices $x,y \in X$.
Recall the subspace $U=U(x,y)$ from Definition  \ref{def:TmodU}.
In this section, we have some comments about $U$.

\begin{proposition} \label{prop:complete} Each irreducible $T$-module with endpoint $1$
is isomorphic to exactly one of the $T$-modules $W_2, W_3, W_4, W_5$ from Theorem
\ref{thm:Udecomp}.
\end{proposition} 
\begin{proof} The irreducible $T$-modules with endpoint $1$ are classified up to isomorphism in \cite[Theorems 2.2, 2.3, 6.1]{hobart}.
The result is immediate from the classification.
\end{proof}

 \begin{proposition} We have
 \begin{align*}
 T(x)\hat y = U = T(y) \hat x.
 \end{align*}
 \end{proposition}
 \begin{proof}  We first show that $T(x){\hat y}=U$. By construction, $U$ is a $T(x)$-module that contains $\hat y$. Therefore $U$ contains $T(x){\hat y}$.
 We assume this containment is proper, and get a contradiction. Recall the orthogonal direct sum $U=\sum_{\ell =1}^5 W_\ell$ from Theorem \ref{thm:Udecomp}.
 Since $\lbrace W_\ell \rbrace_{\ell=1}^5$ are mutually nonisomorphic, there exists a proper subset $K$ of $\lbrace 1,2,3,4,5\rbrace$ such that
 $T(x){\hat y} = \sum_{\ell \in K} W_\ell$. Pick $j\in \lbrace 1,2,3,4,5\rbrace \backslash K$.
 Let $W$ denote the span of the irreducible $T(x)$-modules that are isomorphic to $W_j$.
 Since nonisomorphic irreducible $T(x)$-modules are orthogonal, 
  $W$ is orthogonal to $W_\ell $ for $\ell\in K$. Therefore $W$ is orthogonal to $T(x){\hat y}$. Therefore $W$ is orthogonal to $\hat y$.
 Since $\Gamma$ is distance-transitive, $W$ is orthogonal to ${\hat {\sf y}}$ for all ${\sf y} \in \Gamma(x)$.
 Therefore, $W$ is orthogonal to $E^*_1V$ where $E^*_1=E^*_1(x)$. By construction $W_j \subseteq W$, so $W_j$ is orthogonal to $E^*_1V$.
 By Theorem \ref{thm:Udecomp}, $W_j$ has endpoint $0$ or $1$. In either case $W_j$ is not orthogonal to $E^*_1V$.
 This is a contradiction, so $T(x){\hat y}=U$.
 Interchanging $x \leftrightarrow y$ and using Lemma \ref{lem:Usym}, we obtain $T(y){\hat x}=U$.
  \end{proof}
 
 \noindent We endow the standard module $V$ with the entry-wise product $\circ$. This product is described as follows.
 Pick $u, v \in V$ and write
 \begin{align*}
 u = \sum_{z \in X} u_z {\hat z}, \qquad \qquad v = \sum_{z \in X} v_z {\hat z}, \qquad \qquad u_z, v_z \in \mathbb C.
 \end{align*}
 Then 
 \begin{align*}
 u \circ v = \sum_{z \in X} u_z v_z {\hat z}.
 \end{align*}
 \noindent For any subsets $S_1, S_2 \subseteq X$ we have
 \begin{align} \label{eq:principle}
 \widehat S_1 \circ \widehat S_2 = {\widehat {S_1 \cap S_2}}.
 \end{align}
 
 \begin{proposition}  \label{prop:circ} For $u,v\in U$ we have $u \circ v \in U$.
 \end{proposition}
 \begin{proof} Without loss of generality, we assume that $u,v$ are contained in the $(0/1)$-basis for $U$
 from Definition \ref{def:TmodU}.  Using \eqref{eq:principle} we obtain $u \circ v = \delta_{u,v} u \in U$.
 \end{proof} 
 
 \noindent  With respect to the primitive idempotent $E=E_1$, the graph $\Gamma$ is $Q$-polynomial in the sense of 
 \cite[Definition~11.1]{int}.
  The Norton algebra of $\Gamma$ consists of the vector space $EV$ together
 with a binary operation $\star: EV \times EV \to EV$ such that
 \begin{align*}
 u \star v = E (u \circ v), \qquad \qquad u,v \in EV.
 \end{align*}
 The Norton algebra $EV$ is commutative and nonassociative. See \cite{norton} for more information about the Norton algebra.
 \medskip
 
 \noindent Note that $EU \subseteq U$ since $U$ is a $T$-module.
 
 \begin{proposition} For $u,v\in EU$ we have $u \star v \in EU$. 
 \end{proposition}
 \begin{proof} By Proposition \ref{prop:circ} and since $EU \subseteq U$.
 \end{proof}

  \section{Directions for future research}
 We continue to discuss the bilinear forms graph $\Gamma=(X,\mathcal R)$ and the adjacent vertices $x,y \in X$.
 In this section we give a conjecture and some open problems.
 \medskip
 
 \noindent Recall the subspace $U=U(x,y)$ from Definition  \ref{def:TmodU}, and the primitive idempotent $E=E_1$.

 \begin{conjecture} \label{conj:EU} \rm The following {\rm (i)--(ii)} hold:
 \begin{enumerate}
 \item[\rm (i)] The subspace $EU$ has a basis
 \begin{align*}
 E {\hat x}, \qquad E{\hat y}, \qquad E \widehat O^{(B)}_{1,1}, \qquad  E \widehat O^{(C)}_{1,1}, \qquad  E \widehat O^{(D)}_{1,1}.
 \end{align*}
 \item[\rm (ii)] The Norton subalgebra $EU$ is generated by $E\hat x$ and $E \hat y$.
 \end{enumerate}
 \end{conjecture}
 
 \begin{problem}\rm Referring to  Conjecture \ref{conj:EU}(i), for all basis elements $u,v$ compute
 $u \star v$ as a linear combination of the five basis elements.
 \end{problem}
 
 \begin{problem} \rm Find all the nonzero elements $u \in EU$ such that $u \star u = u$.
 \end{problem}

\section{Acknowledgement} 
\noindent 
This paper was written during J. Williford's sabbatical visit to U. Wisconsin-Madison
(8/17/2025--12/14/2025). During this visit,
J. Williford was supported in part by the Simons Foundation Collaboration Grant 711898.    



\bigskip


\noindent Paul Terwilliger \hfil\break
\noindent Department of Mathematics \hfil\break
\noindent University of Wisconsin \hfil\break
\noindent 480 Lincoln Drive \hfil\break
\noindent Madison, WI 53706-1388 USA \hfil\break
\noindent Email: {\tt terwilli@math.wisc.edu }\hfil\break

\noindent Jason Williford \hfil\break
\noindent  Department of Mathematics and Statistics \hfil\break
\noindent University of Wyoming \hfil\break
\noindent  1000 E. University Ave. \hfil\break
\noindent  Laramie, WY 82071  USA \hfil\break
\noindent Email: {\tt jwillif1@uwyo.edu}\hfil\break

\section{Statements and Declarations}

\noindent {\bf Funding}: The author P. Terwilliger declares that no funds, grants, or other support were received during the preparation of this manuscript. 
The author J. Williford was supported in part by Simons Foundation Collaboration Grant 711898  during the preparation of this manuscript.
\medskip

\noindent  {\bf Competing interests}:  The authors  have no relevant financial or non-financial interests to disclose.
\medskip

\noindent {\bf Data availability}: Details on the combinatorial counts are available upon request.  

\end{document}